\documentclass[12pt]{amsart}
\usepackage[top=4.2cm, bottom=4cm, left=2.4cm, right=2.4cm]{geometry}
\usepackage[utf8]{inputenc}
\usepackage[USenglish]{babel}
\usepackage[T1]{fontenc} 
\usepackage{mathrsfs}
\usepackage{mathtools}
\usepackage{amsmath}
\usepackage{amssymb,epsfig}
\usepackage{caption}
\usepackage{subcaption}
\usepackage{comment}

\usepackage{amsthm}
\usepackage[absolute]{textpos}
\usepackage{mathtools,emptypage}

\usepackage[bookmarks=true]{hyperref}
\usepackage{xcolor}
\hypersetup{
	colorlinks,
	linkcolor={red!50!black},
	citecolor={blue!50!black},
	urlcolor={blue!80!black},
}

\usepackage{tipa} 

\mathtoolsset{showonlyrefs}

\theoremstyle{definition}
\newtheorem{defin}{Definition}[section]
\newtheorem{ex}[defin]{Example}
\theoremstyle{plain}
\newtheorem{theo}[defin]{Theorem}
\newtheorem{lemma}[defin]{Lemma}
\newtheorem{obs}[defin]{Remark}
\newtheorem{prop}[defin]{Proposition}
\newtheorem{cor}[defin]{Corollary}

\newtheorem*{theorem-no-number}{Theorem}
\newtheorem{theorem}{Theorem}

\newcommand{\restr}[1]{\lower3pt\hbox{$|_{#1}$}}

\newcommand{\LocGeod}[1]{\textup{Loc-Geod}(#1)}

\newcommand{\Geod}{\textup{Geod}}

\newcommand{\limi}{\varliminf}
\newcommand{\lims}{\varlimsup}

\renewcommand{\t}[1]{\tilde{#1}}

\newcommand{\N}{\mathbb{N}}

\newcommand{\R}{\mathbb{R}}
\newcommand{\Z}{\mathbb{Z}}
\newcommand{\sfd}{{\sf d}}
\newcommand{\X}{{\rm X}}
\newcommand{\Y}{{\rm Y}}
\newcommand{\CAT}{\textup{CAT}}
\newcommand{\Isom}{\textup{Isom}}
\newcommand{\E}{\mathcal{E}}

\newcommand{\Pack}{\textup{Pack}}
\newcommand{\Cov}{\textup{Cov}}

\newcommand{\MD}{\textup{MD}}
\newcommand{\sfD}{{\sf D}}

\renewenvironment{abstract}
{\par\noindent\textbf{\abstractname.}\ \ignorespaces}
{\par\medskip}

\title{Otal-Peigné's Theorem for Gromov-hyperbolic spaces}
\author{Nicola Cavallucci}
\date{}

\begin{document}
	\maketitle
	\begin{abstract}
		\footnotesize
		We extend the classical Otal-Peigné's Theorem to the class of proper, Gromov-hyperbolic spaces that are line-convex. Namely, we prove that when a group acts discretely and virtually freely by isometries on a metric space in this class then its critical exponent equals the topological entropy of the geodesic flow of the quotient metric space. We also show examples of proper, Gromov-hyperbolic spaces that are not line-convex and for which the statement fails.
	\end{abstract}
	\tableofcontents
	
	\section{Introduction}
	The relation between the topological entropy of the geodesic flow of a Riemannian manifold $(M,g)$ with non-positive sectional curvature and the critical exponent of the discrete group of isometries $\Gamma \cong \pi_1(M)$ acting freely on the universal covering $\t{M}$ of $M$ has been intensively studied during the years.
	A.Manning (\cite{Man79}) proved that in case of compact manifolds with non-positive sectional curvature the topological entropy of the geodesic flow equals the volume entropy of the universal covering, and so it is the critical exponent of the group $\Gamma$. 
	
	We recall that the critical exponent of a group $\Gamma$ acting discretely by isometries on a metric space $(\X,\sfd)$ is classically defined as
	\begin{equation}
		\label{eq:defin_critical_exponent}
		h_\text{crit}(\X;\Gamma) := \lims_{T \to +\infty} \frac{1}{T} \log \#(\Gamma x \cap B(x,T)),
	\end{equation}
	where $x$ is any point of $\X$.
	
	Manning's result has been generalized to compact, locally geodesically complete, locally $\CAT(0)$ spaces by R.Ricks in \cite[Theorem A]{Ric19}.
	\vspace{2mm}
	
	We are interested in the much more complicated case of non-cocompact actions. For groups acting on Riemannian manifolds the following result is due to J.P.Otal and M.Peigné in the torsion-free case (cp. \cite{OP04}, see also \cite{Led13}), and to F.Paulin, M.Pollicott and B.Schapira in the case of groups with torsion (cp. \cite[Theorem 6.1]{PaulinPollicottShapira2015}).
	\begin{theorem-no-number}[Otal-Peigné]
		\label{intro-thm-OP}
		Let $\tilde{M}$ be a complete, simply connected, Riemannian manifold with pinched, negative sectional curvature, i.e. $-b^2 \leq \textup{Sec}_g \leq -a^2 <0$. Let $\Gamma$ be a non-elementary, discrete group of isometries of $\tilde{M}$.
		Then the topological entropy of the geodesic flow on $M:=\Gamma \backslash \tilde{M}$ equals the critical exponent of $\Gamma$. 
	\end{theorem-no-number} 
	
	The price to pay in order to consider every non-elementary group acting discretely on $\tilde{M}$ is the condition on the sectional curvature. While the lower bound is quite natural, since every compact manifold has such a bound, the negative upper bound marks a difference among the cocompact case and the general one. 
	\vspace{2mm}
	
	The purpose of this paper is to extend Otal-Peigné's Theorem to a wider class of metric spaces by weakening the regularity of the space and the upper curvature bound and by essentially removing the lower bound on the curvature. 
	
	The upper curvature bound is replaced by Gromov hyperbolicity. This class is extremely wild and contains many non-manifold examples. Moreover, this condition is different in nature from a negative upper bound on the sectional curvature, or a CAT$(-1)$-condition, since it does not give any information on the local geometry of the metric space. 
	
	For Gromov-hyperbolic metric spaces $\X$ and discrete groups of isometries $\Gamma < \Isom(\X)$, we need to clarify the definition of the quotient geodesic flow of $\Gamma \backslash \X$. We denote by $\Geod(\X)$ the space of all geodesic lines of $\X$. It supports a natural action $\Phi_t$ by $\R$ via reparametrization. The dynamical system $(\Geod(\X), \Phi_1)$ is called the geodesic flow of $\X$. Every isometry of $\X$ acts on $\Geod(\X)$ and commutes with the flow $\Phi_t$. The quotient geodesic flow appearing in Theorem \ref{theo:intro-equality-Gromov-hyperbolic} is the dynamical system $(\Gamma \backslash \Geod(\X), \bar{\Phi}_1)$, where $\bar{\Phi}_1$ is the induced action on the quotient space $\Gamma \backslash \Geod(\X)$.
	
	This definition, proposed for instance in  \cite{Rob03,DilsavorThompson2023}, may seem artificial and not related to the classical notion of geodesic flow of $\Gamma \backslash \X$. However, if $\X$ is $\CAT(-1)$, or even Busemann convex, and $\Gamma$ is torsion-free then the dynamical system $(\Gamma \backslash \Geod(\X), \bar{\Phi}_1)$ is conjugated to $(\LocGeod{\Gamma \backslash \X}, \Phi_1)$, where $\LocGeod{\Gamma \backslash \X}$ is the space of local geodesics of $\Gamma \backslash \X$ and $\Phi_1$ is again the standard reparametrization. This result is proven in Corollary \ref{cor:conjugation_of_flows}.
	For more details we refer to Section \ref{sec:geodesic_flow}.
	
	The topological entropy of the quotient geodesic flow of $\Gamma\backslash \X$ is then defined as the topological entropy of the dynamical system $(\Gamma \backslash \Geod(\X), \bar{\Phi}_1)$. By the variational principle, it can be expressed in the following form:
	\begin{equation}
		\label{eq:defin_h_top_with_measures}
		h_\text{top} := \sup_{\mu \in \E_1}h_\mu,
	\end{equation}
	where $\E_1$ is the set of ergodic invariant measures of the dynamical system and $h_\mu$ is the Kolmogorov-Sinai entropy of $\mu$. We refer to Section \ref{sec:topological_entropy} for the definitions.
	
	\vspace{2mm}
	The equivalent of Otal-Peigné's Theorem would then be the following: for every Gromov-hyperbolic metric space $\X$ and every discrete $\Gamma < \Isom(\X)$, the critical exponent of $\Gamma$ equals the topological entropy of the quotient geodesic flow of $\Gamma \backslash \X$.
	
	However, Gromov hyperbolicity alone is too weak, namely there are Gromov hyperbolic spaces for which this version of Otal-Peigné's Theorem does not hold (see Theorem \ref{theo:intro_counterexample_h_top>h_Gamma}). 
	
	We need to introduce an additional weak upper bound on the curvature through the next definition: a metric space $(\X,\sfd)$ is said to be \emph{line-convex} if for every two geodesic lines $\gamma,\gamma'\colon \R \to \X$, the function $t\mapsto \sfd(\gamma(t),\gamma'(t))$ is convex. For instance, every Busemann convex space is line-convex, but the vice versa is false, see for instance the space of Theorem \ref{theo:intro-weird-example}, which is line-convex but not simply connected. In general also the line-convexity property does not give control on the local geometry of the metric space, since it is preserved by local perturbations of the metric space that are not seen by any geodesic line.

	The main result of this work is the following.
	\begin{theorem}
		\label{theo:intro-equality-Gromov-hyperbolic}
		Let $\X$ be a proper, geodesic, line-convex, Gromov hyperbolic metric space. Let $\Gamma < \Isom(\X)$ be discrete, virtually torsion-free and non-elementary. Assume that $h_\textup{erg}(\X;\Gamma) < +\infty$. Then the topological entropy of the quotient geodesic flow of $\Gamma \backslash \X$ equals the critical exponent of $\Gamma$.
	\end{theorem}

	The quantity $h_\textup{erg}(\X;\Gamma)$ is what we call the \emph{ergodic entropy} of the action of $\Gamma$ on $\X$. It is defined as follows:
	\begin{equation}
		\label{eq:defin_h_erg_intro}
		h_\textup{erg}(\X;\Gamma) := \lim_{r\to 0} \lims_{T \to +\infty} \frac{1}{T} \log \text{Cov}(\overline{B}(x,T);\Lambda_\text{erg};r),
	\end{equation}
	where $x$ is a fixed basepoint of $\X$ and $\text{Cov}(\overline{B}(x,T);\Lambda_\text{erg},r)$ denotes the minimal number of $r$-balls needed to cover the set of points that belong to $\overline{B}(x,T)$ and that lie in a geodesic line with endpoints in the ergodic limit set $\Lambda_\text{erg}(\Gamma) \subseteq \partial \X$.
	
	The ergodic limit set of the group $\Gamma$ has been introduced in \cite{Cav24}. A point $z\in \partial \X$ belongs to $\Lambda_\text{erg}(\Gamma)$ if a geodesic ray $\xi_{x,z}$ joining the basepoint $x$ to $z$ satisfies the following property: for some $\tau > 0$ and some increasing sequence of real numbers $\vartheta_i$ such that $\exists \lim_{i\to +\infty} \frac{\vartheta_i}{i} \in (0,+\infty)$ it holds that for every $i\in \N$ there exists $t_i \in [\vartheta_i, \vartheta_{i+1}]$ such that $\sfd(\xi_{x,z}(t_i), \Gamma x) \le \tau$. For instance, if the sequence $\vartheta_i$ is of the form $i\cdot \sigma$ then $z$ belongs to the set of uniformly radial limit points and the projection of a geodesic ray $\xi_{x,z}$ on $\Gamma \backslash \X$ is entirely contained in a compact subset of $\Gamma \backslash \X$. The condition defining $\Lambda_\text{erg}(\Gamma)$ is a generalization of uniformly radial points, and if $z\in \Lambda_\text{erg}(\Gamma)$ then the projection of a geodesic ray $\xi_{x,z}$ on $\Gamma \backslash \X$ may be not contained in a compact region, but it returns to a fixed compact region infinitely many times and the sequence of returning times is asymptotically linear. 
	
	From a dynamical point of view, this condition comes directly from Birkhoff's ergodic Theorem, and indeed every ergodic measure on the quotient geodesic flow of $\Gamma \backslash \X$ is concentrated on classes of geodesics with endpoints in $\Lambda_\text{erg}(\Gamma)$, see Proposition \ref{prop:Birkhoff_mean} and \cite[Theorem C]{Cav24}.
	
	The assumption $h_\textup{erg}(\X;\Gamma) < +\infty$ appearing in Theorem \ref{theo:intro-equality-Gromov-hyperbolic} plays the role of a very weak lower bound on the curvature. For instance it is always satisfied when $\Gamma$ is quasiconvex cocompact, see Corollary \ref{cor:h_crit=h_erg}. Also, it is always satisfied under other weak geometric constraints. A metric space $\X$ has \emph{bounded geometry} if 
	\begin{equation}
		\label{eq:defin_bounded_geometry_intro}
		\sup_{x\in \X} \Cov(\overline{B}(x,R),r) < + \infty\quad \text{for every } 0<r\le R,
	\end{equation}
	where $\Cov(\overline{B}(x,R),r)$ denotes the minimal number of balls of radius $r$ needed to cover the ball $\overline{B}(x,R)$. This assumption is extremely weak: notice that the supremum in \eqref{eq:defin_bounded_geometry_intro} is always finite on every compact region of $\X$, so the definition asks just for a uniform bound in the whole space. If a Gromov-hyperbolic metric space has bounded geometry then the ergodic entropy of $\Gamma$ coincides with the Minkowski dimension of the limit set of $\Gamma$ and this dimension is finite (cp. Corollary \ref{cor:bounded_geometry_then_finite_erg_entropy}). Therefore Theorem \ref{theo:intro-equality-Gromov-hyperbolic} applies to line-convex, Gromov-hyperbolic metric spaces with bounded geometry. Conditon \eqref{eq:defin_bounded_geometry_intro} can be further weakened as we show in Section \ref{sec:preliminaries}.
	
%
%
	The condition  $h_\text{erg}(\X;\Gamma) < \infty$ is not too far from the optimal $h_\text{crit}(\X;\Gamma) < \infty$. Indeed, at least for spaces with bounded geometry, $h_\text{erg}(\X;\Gamma)$ equals the Minkowski dimension of the limit set of $\Gamma$, while $h_\text{crit}(\X;\Gamma)$ equals the Hausdorff dimension of the radial limit set by the classical Bishop-Jones' Theorem (cp. \cite{BJ97}, \cite{Pau97}, \cite{DSU17}).
	We refer to Section \ref{sec:preliminaries} for more details and examples.
	\vspace{2mm}
	
	The proof of the original Otal-Peigné's Theorem, also in \cite{PaulinPollicottShapira2015}, is based on local arguments where the negative sectional curvature assumption is exploited. The proof of the inequality $h_\text{top}(\Gamma \backslash \Geod(\X), \bar{\Phi}_1) \le h_\text{crit}(\X;\Gamma)$ has been extended, with the same techniques, to geodesically complete, CAT$(-1)$ spaces in \cite{DilsavorThompson2023}, assuming the finiteness of the Minkowski dimension of the space $\Gamma \backslash \Geod(\X)$. This condition, somehow transversal to our $h_\textup{erg}(\X;\Gamma) < \infty$, is more sensible to the local geometry of $\X$.
	
	This classical approach cannot work in the setup of Theorem \ref{theo:intro-equality-Gromov-hyperbolic} due to the lack of control on the local geometry of the spaces under consideration. Therefore we will follow a completely new approach involving asymptotic arguments only. 
	This new strategy allows to prove the inequality $h_\text{top}(\Gamma \backslash \Geod(\X), \bar{\Phi}_1)\le h_\text{crit}(\X;\Gamma)$ also for groups with arbitrary torsion. On the other hand, it really needs the line-convexity of $\X$ in order to work.
	\begin{theorem}
		\label{theo:intro_counterexample_h_top>h_Gamma}
		There exists a proper, geodesic, Gromov-hyperbolic metric space $\X$ with bounded geometry, but not line-convex, and a discrete, non-elementary, cocompact, torsion-free group $\Gamma < \Isom(\X)$ such that $$h_\textup{top}(\Gamma \backslash \Geod(\X), \bar{\Phi}_1) > h_\textup{crit}(\X;\Gamma).$$
	\end{theorem}
%

	Our asymptotic approach gives the proof of the inequality $h_\text{top}(\Gamma \backslash \Geod(\X), \bar{\Phi}_1) \ge h_\text{crit}(\X; \Gamma)$ for every Gromov hyperbolic space, without further assumptions.
	\begin{theorem}
		\label{theo:intro-h_top>=h_Gamma}
		Let $\X$ be a proper, geodesic, Gromov-hyperbolic metric space and let $\Gamma < \Isom(\X)$ be discrete, virtually torsion-free and non-elementary. Then 
		$$h_\textup{top}(\Gamma \backslash \Geod(\X), \bar{\Phi}_1) \ge h_\textup{crit}(\X;\Gamma).$$
	\end{theorem}

	We do not know if Theorem \ref{theo:intro-h_top>=h_Gamma} is true for every discrete group with arbitrary torsion, as it is in the manifold case because of Paulin-Pollicott-Shapira's version of the original Otal-Peigné's Theorem.

	In this regard, one could try, as suggested in Remark 6.1 of \cite{DilsavorThompson2023}, to adapt the proof of Theorem \ref{theo:intro-h_top>=h_Gamma} given in \cite[Lemma 6.7]{PaulinPollicottShapira2015} to the $\CAT(-1)$ setting. However, the proof of \cite[Lemma 6.7]{PaulinPollicottShapira2015} is based on \cite[Proposition 4.9]{PaulinPollicottShapira2015}, which uses the lower bound on the sectional curvature. However it could be possible to extend the proof of \cite[Proposition 4.9]{PaulinPollicottShapira2015} to geodesically complete, packed, CAT$(-1)$ spaces following the ideas of \cite{CavS20bis} in order to produce free subgroups of a torsion-free group of isometries.
	It is not clear to the author how to get Theorem \ref{theo:intro-h_top>=h_Gamma} for every group with arbitrary torsion, even in the simplified setting of geodesically complete, CAT$(-1)$ spaces.
	\vspace{2mm}
	
	Interesting geometric phenomena can occur. 	
	Let $\X$ be a $\CAT(-1)$ space and let $\Gamma < \Isom(\X)$ be discrete, non-elementary and virtually torsion-free. Let $\Gamma_0$ be a torsion-free, finite index, normal subgroup of $\Gamma$. As stated above, the two flows $(\Gamma_0\backslash \Geod(\X), \bar{\Phi}_1)$ and $(\LocGeod{\Gamma_0\backslash \X}, \Phi_1)$ are conjugated. Moreover, the finite group $\Gamma / \Gamma_0$ acts naturally on both spaces commuting with the flows. On $(\Gamma_0\backslash \Geod(\X), \bar{\Phi}_1)$ it acts via the quotient action of the global action of $\Gamma$ on $\Geod(\X)$, and the quotient is exactly $(\Gamma\backslash\Geod(\X), \bar{\Phi}_1)$. On $(\LocGeod{\Gamma_0\backslash \X}, \Phi_1)$ it acts via the natural action on local geodesics induced by $\Gamma / \Gamma_0$ as group of isometries of $\Gamma_0 \backslash \X$. The resulting flow is $((\Gamma / \Gamma_0) \backslash \LocGeod{\Gamma_0 \backslash \X}, \bar{\Phi}_1)$. 
	
	The next result shows that these actions do not commute, and the flows $(\Gamma\backslash\Geod(\X), \bar{\Phi}_1)$ and $((\Gamma / \Gamma_0) \backslash \LocGeod{\Gamma_0 \backslash \X}, \bar{\Phi}_1)$ are in general not conjugated as they can have different topological entropy.
	
	\begin{theorem}
		\label{theo:intro-counter-tree}
		Let $T_4$ be the regular tree of valency $4$ and let $\Gamma_0$ be the free group on two generators with its standard action on $T_4$. Then there exists a discrete, cocompact group $\Gamma < \Isom(T_4)$ containing $\Gamma_0$ as finite index normal subgroup such that
		\begin{itemize}
			\item[(i)] $h_\textup{top}(\Gamma \backslash \Geod(T_4), \bar{\Phi}_1) = \log 3$;
			\item[(ii)] $h_\textup{top}((\Gamma/\Gamma_0) \backslash \LocGeod{\Gamma_0 \backslash T_4}, \bar{\Phi}_1) = 0$.
		\end{itemize}
	\end{theorem}

%
	
	Another interesting phenomenon is the following.
	\begin{theorem}
		\label{theo:intro-weird-example}
		There exists a proper, geodesic, line-convex, Gromov-hyperbolic metric space $\X$ with bounded geometry and a discrete, torsion-free, cocompact group $\Gamma < \Isom(\X)$ such that:
		\begin{itemize}
			\item[(i)] $\Gamma \backslash \X$ is a locally $\CAT(-1)$, locally geodesically complete, compact metric space. However, $\X$ is not simply connected, so it is not the universal cover of $\Gamma \backslash \X$, and $\Gamma \backslash \Geod(\X)$ is a proper subset of $\LocGeod{\Gamma \backslash \X}$;
			\item[(ii)] it holds that $$h_\textup{top}(\LocGeod{\Gamma\backslash \X}, \Phi_1) = h_\textup{crit}(\tilde{\X};\pi_1(\Gamma\backslash \X)) > h_\textup{crit}(\X;\Gamma) = h_\textup{top}(\Gamma \backslash \Geod(\X), \bar{\Phi}_1),$$
			where $\tilde{\X}$ is the universal cover of $\Gamma \backslash \X$.
		\end{itemize}
	\end{theorem}
	
	A last comment on the second part of the original version of Otal-Peigné's Theorem for manifolds, that states that there is at most one measure of maximal entropy, i.e. realizing the supremum in \eqref{eq:defin_h_top_with_measures}. This is not true in general for proper, Gromov-hyperbolic spaces that are line-convex.
	\begin{ex}
		Let $\Gamma < \Isom(\mathbb{H}^2)$ be discrete, cocompact and torsion-free, so $M=\Gamma \backslash \mathbb{H}^2$ is a compact hyperbolic surface. Let $Q$ be any compact geodesic metric space and let $\X := \mathbb{H}^2 \times Q$ equipped with the product metric. Then $\X$ is proper and geodesic. It is also Gromov-hyperbolic since it is at bounded distance from $\mathbb{H}^2$. Since every geodesic on $\X$ is the product of geodesics of the factors (cp. \cite[Proposition I.5.13]{BH09}), we have that every geodesic line of $\X$ is of the form $t\mapsto (\gamma(t),q)$ where $\gamma$ is a geodesic line in $\mathbb{H}^2$ and $q \in Q$, i.e. $\Geod(\X) = \Geod(\mathbb{H}^2) \times Q$. This shows that $\X$ is line-convex. Observe that $\Gamma$ acts discretely, cocompactly and freely on $\X$ by acting on $\mathbb{H}^2$ as above and trivially on $Q$. Theorem \ref{theo:intro-equality-Gromov-hyperbolic}, that we can apply by Corollary \ref{cor:h_crit=h_erg}, implies that $h_\textup{top}(\Gamma \backslash \Geod(\X), \bar{\Phi}_1) = h_\textup{crit}(\X; \Gamma) = h_\textup{crit}(\mathbb{H}^2; \Gamma) = 1$. Let $\mu$ be the unique measure of maximal entropy of $\Gamma \backslash \Geod(\mathbb{H}^2)$ provided by Otal-Peigné's Theorem for manifolds. Observe that $h_\mu = h_\textup{top}(\Gamma \backslash \Geod(\mathbb{H}^2), \bar{\Phi}_1) = h_\textup{crit}(\mathbb{H}^2; \Gamma) = 1$. For every $q\in Q$ let $\delta_q$ be the Dirac measure at $q$ and define $\nu_q := \mu \otimes \delta_q$ on $\Gamma \backslash \Geod(\X) = \Gamma \backslash \Geod(\mathbb{H}^2) \times Q$. A direct computation shows that $\nu_q$ is $\bar{\Phi}_1$-invariant and that $h_{\nu_q} \ge 1$, implying that $\nu_q$ realizes $h_\textup{top}(\Gamma \backslash \Geod(\X), \bar{\Phi}_1)$ for every $q\in Q$. Observe that these measures are concentrated on different subsets, that are necessarily disjoint since every such measure is ergodic.
	\end{ex}
	
	\subsection{Organization of the paper} In Section \ref{sec:preliminaries} we will recall the basic definitions and we will discuss and compare in detail several notions of lower bounds on the curvature for arbitrary metric spaces and for Gromov-hyperbolic ones. The goal is to give a better intuition of the ergodic entropy appearing in Theorem \ref{theo:intro-equality-Gromov-hyperbolic}.
	
	In Section \ref{sec:geodesic_flow} we will recall basic properties of the space of geodesic lines of a metric space and we will introduce the geodesic flow. In Section \ref{sec:topological_entropy} we will discuss several notions of entropies of a dynamical system, with applications to the case of the geodesic flow.
	Section \ref{sec:Lip-Bowen-entropy} is the last preparatory part: here we introduce the notion of $f$-entropy of a subset of the boundary at infinity of a Gromov-hyperbolic space and we compare it to other classical quantities. This section completes the study of Section \ref{subsec:geodesic_entropy_subsets}.
	
	Finally, in the last three sections we will prove the main theorems.

	\section{Preliminaries}
	\label{sec:preliminaries}
	We begin by recalling some notations. Let $(\X,\sfd)$ be a metric space. The \emph{open (resp.closed) ball} of radius $r$ and center $x \in X$ is denoted by $B(x,r)$ (resp. $\overline{B}(x,r)$). If we want to specify the metric we write $B_\sfd(x,r)$ (resp. $\overline{B}_\sfd(x,r))$. A metric space is \emph{proper} if every closed ball is compact.
	
	A \emph{geodesic} is an isometric embedding $\gamma\colon I \to \X$ where $I=[a,b]$ is a a bounded interval of $\mathbb{R}$. The points $\gamma(a), \gamma(b)$ are called the endpoints of $\gamma$. A metric space $\X$ is called geodesic if for every couple of points $x,y\in \X$ there exists a geodesic whose endpoints are $x$ and $y$. Every such geodesic will be denoted, with an abuse of notation, by $[x,y]$. A \emph{geodesic ray} is an isometric embedding $\xi\colon[0,+\infty)\to \X$ while a \emph{geodesic line} is an isometric embedding $\gamma\colon \mathbb{R}\to \X$. 
	A map $\gamma \colon I \to \X$, where $I$ is an interval of $\mathbb{R}$, is a \emph{local geodesic} if for every $t\in I$ there exists $\varepsilon > 0$ such that $\gamma\vert_{[t-\varepsilon, t+\varepsilon]}$ is a geodesic. The space of local geodesic lines, namely the set of local geodesics of $\X$ defined over $\mathbb{R}$, is denoted by $\LocGeod{\X}$, while its subspace of geodesic lines is denoted by $\Geod(\X)$. 
	
	The \emph{evaluation map at time $s \in \R$} is the function $e_s\colon \LocGeod{\X} \to \X$, $\gamma \mapsto \gamma(s)$.
	
	A metric space is \emph{geodesically complete} if every geodesic can be extended to a geodesic line. It is \emph{locally geodesically complete} if every local geodesic can be extended to a local geodesic line. A metric space $(\X,\sfd)$ is \emph{line-convex} if for every two geodesic lines $\gamma,\gamma'\colon \R \to \X$, the function $t\mapsto \sfd(\gamma(t),\gamma'(t))$ is convex.
	
	Let $\X$ be a metric space, $B\subseteq \X$ and $r>0$. A subset $S$ of $B$ is called {\em $r$-separated} if $\sfd(y,y') > r$ for all $y,y'\in S$, while it is called {\em $r$-dense} if for all $y \in B$ there exists $z\in S$ such that $\sfd(y, z) \le r$.
	We define $\Pack_\sfd(B,r) \in \N\cup \{\infty\}$ as the maximal cardinality of a $r$-separated subset of $B$ and $\Cov_\sfd(B,r) \in \N\cup \{\infty\}$ as the minimal cardinality of a $r$-dense subset. When the metric $\sfd$ is clear from the context, we just omit it. The two quantities above are classically related, more precisely:
	\begin{equation}
		\label{eq:Cov_Pack}
		\Pack_\sfd(B,2r) \le \Cov_\sfd(B,2r) \le \Pack_\sfd(B,r)
	\end{equation}
	for every $B\subseteq \X$ and every $r>0$.
	
	\subsection{Bounded geometry, metric doubling, covering entropy and geodesic entropy}
	\label{subsec:geodesic_entropy}
	Let $\X$ be a proper metric space. There are several notions that can be used to control the geometry of $\X$. In \eqref{eq:defin_bounded_geometry_intro} we introduced the notion of \emph{bounded geometry}, where we ask that 
	\begin{equation}
		\label{eq:defin_bounded_geometry}
		\sup_{x\in \X} \Cov_\sfd(\overline{B}(x,R),r) =: C_{R,r} < +\infty \quad \text{for every  } 0<r\le R.
	\end{equation}
	
	A related notion is the metric doubling condition. We say that $\X$ is \emph{metrically doubling up to scale $r_0$} if there exists $C_D \ge 1$ such that $\Pack_\sfd(\overline{B}(x,2r), r) \le C_D$ for every $x\in \X$ and every $0<r\le r_0$. Here the choice of the packing function instead of the covering one is not important in view of \eqref{eq:Cov_Pack}: it has been done to precisely cite the following result.
	\begin{lemma}
		\label{lemma:metrically_doubling_basics}
		Let $\X$ be a proper, geodesic metric space that is metrically doubling up to scale $r_0$ with constant $C_D$. Then for every $x\in \X$ and every $0<r\le R$ it holds that
		$$\Pack_\sfd(\overline{B}(x,R),r) \le C_D^2(1+C_D^2)^{\frac{R}{\min\{r,r_0\}} - 1}.$$
	In particular $\X$ has bounded geometry.
	\end{lemma}
	\begin{proof}
		A space as in the assumptions satisfies $\Pack_\sfd(\overline{B}(x,3r), r) \le C_D^2$ for every $x\in \X$ and every $r\le \frac{r_0}{2}$. Therefore the thesis follows by \cite[Lemma 4.7]{CavS20}.
	\end{proof}

	Another possible bound on the geometric complexity of $\X$ is via its \emph{covering entropy}. It is defined as
	$$h_\Cov(\X) := \lim_{r\to 0} \lims_{T \to +\infty} \frac{1}{T} \log \Cov_\sfd(\overline{B}(x,T), r).$$
	This quantity does not depend on the choice of the basepoint $x\in \X$. Moreover one can replace $\Cov_\sfd(\overline{B}(x,T), r)$ with $\Pack_\sfd(\overline{B}(x,T), r)$ without changing the value of the limit, because of \eqref{eq:Cov_Pack}. Every space which is metrically doubling up to some scale has finite covering entropy.
	\begin{lemma}
		\label{lemma:bounded_geometry_no_r}
		Let $\X$ be a proper, geodesic metric space with bounded geometry. Then the function 
		$$r\mapsto \lims_{T \to +\infty} \frac{1}{T} \log \Cov_\sfd(\overline{B}(x,T), r)$$ 
		is constant. Moreover, if $\X$ is metrically doubling up to scale $r_0$ with constant $C_D$ then $h_\Cov(\X) \le \frac{\log(1+C_D^2)}{r_0} < +\infty$.
	\end{lemma}
	\begin{proof}
		 The function $r\mapsto \lims_{T \to +\infty} \frac{1}{T} \log \Cov_\sfd(\overline{B}(x,T), r)$ is decreasing, se we just need to check one inequality.	Let $0<r\le R$.	Let $C_{R,r} < +\infty$ be the constant in \eqref{eq:defin_bounded_geometry}. Then
		$$\Cov_\sfd(\overline{B}(x,T), r) \le \Cov_\sfd(\overline{B}(x,T), R) \cdot C_{R,r}$$
		for every $T>0$. This implies the first part of the thesis. As a consequence we have that $h_\Cov(\X) \le \lims_{T \to +\infty} \frac{1}{T} \log C_{T,r}$ for every $r>0$. The right hand side may be finite or not, depending on the behaviour of the function $T \mapsto C_{T,r}$. Now, if in addition $\X$ is metrically doubling up to scale $r_0$ with constant $C_D$, then $C_{T,2r_0} \le C_D^2(1+C_D^2)^{\frac{T}{r_0} - 1}$ because of Lemma \ref{lemma:metrically_doubling_basics} and \eqref{eq:Cov_Pack}. This gives the second part of the statement.
	\end{proof}

	The covering entropy can be finite also for spaces that have no bounded geometry.
	
	\begin{ex}
		Let $\{d_n\}_{n\in \N}$ be a sequence of integers. Let $\X$ be the space obtained by taking $\R$ and gluing $d_n$ segments of length $1$ at the point $n\in \N \subseteq \R$. For every fixed $r>0$ we need at most $(n + \sum_{j=1}^{n-1} d_j) \cdot \frac{1}{r}$ points to cover the ball $\overline{B}(0,n)$. Choosing for instance $d_n = n$ we obtain that $h_\text{Cov}(\X) = 0$. On the other hand, the same choice of $d_n$ implies that $\sup_{n\in \N} \Cov_\sfd(\overline{B}(n,1), \frac{1}{4}) = +\infty$.
	\end{ex}

	One can weaken the condition $h_\text{Cov}(\X) < +\infty$ using the \emph{geodesic covering entropy} of $\X$. Its definition is similar to the one of the covering entropy, but we ask to cover only the points of $\overline{B}(x,T)$ that lie in some geodesic line. More precisely,
	$$h_\Geod(\X) := \lim_{r\to 0} \lims_{T \to +\infty} \frac{1}{T} \log \Cov_\sfd(\overline{B}(x,T) \cap e_0(\Geod(\X)), r).$$
	By definition, $h_\Geod(\X) \le h_\Cov(\X)$. The inequality may be strict.
	\begin{ex}
		Let $\{d_n\}_{n\in \N}$ be a sequence of integers. Let $\X$ be the space obtained by taking $\R$ and gluing $d_n$ segments of length $1$ at the point $n\in \N \subseteq \R$.
		If $r\le \frac{1}{2}$ then $\Cov(\overline{B}(0,n), r) \ge n + \sum_{j=1}^{n-1} d_j$. So $h_\Cov(\X) \ge \lims_{n \to +\infty} \frac{\log(\sum_{j=1}^{n-1} d_j)}{n}$. By choosing the sequence $\{d_n\}$ appropriately we can get any value in $[0,+\infty]$. On the other hand, every geodesic line of $\X$ must live in the copy of $\R$ inside $\X$. So, $h_\Geod(\X) = 0$ for every choice of $\{d_n\}$. Observe that $\X$ is line-convex.
	\end{ex}

	Other bounds on the geometric complexity will be presented in Section \ref{subsec:geodesic_entropy_subsets} in case of Gromov-hyperbolic spaces. They would resemble the bounds we gave in this section.
	
	The condition appearing in Theorem \ref{theo:intro-equality-Gromov-hyperbolic} is weaker than all the conditions we discussed in this section.
		
	\subsection{Gromov hyperbolic spaces}
	References for this subsection are \cite{CDP90, BH09}.
	Let $X$ be a metric space and let $x,y,z \in \X$. The {\em Gromov product} of $y$ and $z$ with respect to $x$ is
	$$(y,z)_x = \frac{1}{2}\big( d(x,y) + d(x,z) - d(y,z) \big).$$
	
	$\X$ is called {\em $\delta$-hyperbolic} if   for every $x,y,z,w \in \X$   the following {\em 4-points condition} hold:
	\begin{equation}\label{hyperbolicity}
		(x,z)_w \geq \min\lbrace (x,y)_w, (y,z)_w \rbrace -  \delta 
	\end{equation}
	or, equivalently,
	\begin{equation}
		\label{four-points-condition}
		d(x,y) + d(z,w) \leq \max \lbrace d(x,z) + d(y,w), d(x,w) + d(y,z) \rbrace + 2\delta. 
	\end{equation}
	
	$\X$ is   {\em Gromov hyperbolic} if it is $\delta$-hyperbolic for some $\delta \geq 0$. Gromov-hyperbolicity should be considered as a negative-curvature condition at large scale: for instance every CAT$(\kappa)$ metric space, with $\kappa <0$ is $\delta$-hyperbolic for a constant $\delta$ depending only on $\kappa$. 
	
	Let $\X$ be a proper, Gromov-hyperbolic metric space and let $x$ be a point of $\X$. The {\em Gromov boundary} of $\X$ is defined as the quotient 
	$$\partial \X := \left\lbrace (z_n)_{n \in \mathbb{N}} \subseteq \X \,:\,   \lim_{n,m \to +\infty} (z_n,z_m)_{x} = + \infty \right\rbrace \hspace{1mm} /_\sim,$$
	where $(z_n)_{n \in \mathbb{N}}$ is a sequence of points in $\X$ and $\sim$ is the equivalence relation defined by $(z_n)_{n \in \mathbb{N}} \sim (z_n')_{n \in \mathbb{N}}$ if and only if $\lim_{n,m \to +\infty} (z_n,z_m')_{x} = + \infty$.  
	We will write $ z := [(z_n)] \in \partial \X$ for short, and we say that $(z_n)$ {\em converges} to $z$. This definition does not depend on the basepoint.
	
	There is a natural topology on $\X\cup \partial \X$ that extends the metric topology of $\X$. 
	
	Every geodesic ray $\xi$ defines a point  $\xi^+:=[(\xi(n))_{n \in \mathbb{N}}]$  of the Gromov boundary $ \partial \X$: we  say that $\xi$ {\em joins} $\xi(0)$ {\em to} $\xi^+$. Moreover for every $z\in \partial \X$ and every $x\in \X$ it is possible to find a geodesic ray $\xi$ such that $\xi(0)=x$ and $\xi^+ = z$. Indeed if $(z_n)$ is a sequence of points converging to $z$ then, by properness of $\X$, a sequence of geodesics $[x,z_n]$ subconverges to a geodesic ray $\xi$ which has the properties above (cp. \cite[Lemma III.3.13]{BH09}). A geodesic ray joining $x$ to $z\in \partial \X$ will be denoted by $\xi_{x,z}$.
	
	The relation between Gromov product and geodesic rays is highlighted in the following lemma.
	\begin{lemma}[\text{\cite[Lemma 4.2]{Cav21ter}}]
		\label{product-rays}
		Let $\X$ be a proper, $\delta$-hyperbolic metric space, $z,z'\in \partial \X$, $x\in \X$ and $b>0$. Then for all geodesic rays $\xi_{x,z}, \xi_{x,z'}$ it holds that
		\begin{itemize}
			\item[(i)] if $(z,z')_{x} \geq T$ then $\sfd(\xi_{x,z}(T - \delta),\xi_{x,z'}(T - \delta)) \leq 4\delta$;
			\item[(ii)] if $\sfd(\xi_{x,z}(T),\xi_{x,z'}(T)) < 2b$ then $(z,z')_{x} > T - b$.
		\end{itemize}
	\end{lemma}
	
	
	The following is a standard computation, see for instance \cite[Proposition 8.10]{BCGS}.
	\begin{lemma}
		\label{parallel-geodesics}
		Let $\X$ be a proper, $\delta$-hyperbolic metric space. Let $\xi, \xi'$ be two geodesic rays with $\xi^+ = \xi'^+$. Then there exist $t_1,t_2\geq 0$ such that $t_1+t_2=\sfd(\xi(0),\xi'(0))$ and  $\sfd(\xi(t + t_1),\xi'(t+t_2)) \leq 8\delta$ for all $t\in \mathbb{R}$.
	\end{lemma}

	Given a subset $C \subseteq \partial \X$, we define the \emph{quasi-convex hull} of $C$, and we denote it by $\textup{QC-Hull}(C) \subseteq \X$, as the union of all geodesic lines with endpoints in $C$. We will need the following basic, but important, facts.
	
	\begin{lemma}[{\cite[Lemma 5.4]{Cav21}}]
		\label{lemma:approximation-ray-line}
		Let $\X$ be a proper, geodesic, $\delta$-hyperbolic space.
		Let $x\in \X$ and $C\subseteq \partial \X$ be a subset with $\# C \ge 2$. Then for every $z\in C$ there exists a geodesic line $\gamma$ with endpoints in $C$ such that for every geodesic ray $\xi_{x,z}$ it holds that $\sfd(\xi_{x,z}(t), \gamma(t))\leq 22\delta + \sfd(x,\textup{QC-Hull}(C))$ for every $t\geq 0$.
	\end{lemma}

	\begin{lemma}
		\label{lemma:projection_quasigeodesic}
		Let $\X$ be a proper, geodesic, $\delta$-hyperbolic space and let $x\in \X$. Let $\gamma$ be a geodesic line of $\X$ parametrized in such a way that $\gamma(0)$ is a point of $\gamma$ which is closest to $x$. Let $\xi_{x,\gamma^+}$ be a geodesic ray. For every $t\ge 0$ let $y_t$ be the unique point on $\xi_{x,\gamma^+}$ such that $\sfd(x,y_t) = \sfd(x,\gamma(t))$. Then $\sfd(\gamma(t), y_t) \le 144\delta$.
	\end{lemma}
	\begin{proof}
		Let $t\ge 0$. By \cite[Lemma 5.2]{Cav21} we can find a point $y'$ on $\xi_{x,\gamma^+}$ such that $\sfd(\gamma(t),y') \le 72\delta$, and so $\sfd(x,\gamma(t)) - 72\delta \le \sfd(x,y') \le \sfd(x,\gamma(t)) + 72\delta$. Then $\sfd(y_t,y') \le 72\delta$, so $\sfd(\gamma(t),y_t) \le 144\delta$.
	\end{proof}
	
	Let $x\in \X$ be a basepoint.
	As in \cite{Pau96} and \cite{Cav21ter} we define the {\em generalized visual ball of center $z \in \partial \X$ and radius $\rho \geq 0$} as
	$$B_x(z,\rho) = \bigg\lbrace z' \in \partial \X \,:\, (z,z')_{x} > \log \frac{1}{\rho} \bigg\rbrace.$$
	
	Generalized visual balls are related to shadows. The \emph{shadow of radius $r>0$ casted by a point $y\in X$ with center $x$} is the set
	$$\text{Shad}_x(y,r) = \lbrace z\in \partial \X \,:\, \xi_{x,z}\cap B(y,r) \neq \emptyset \text{ for every geodesic ray } \xi_{x,z}\rbrace.$$
	\begin{lemma}[\text{\cite[Lemma 4.8]{Cav21ter}}]
		\label{shadow-ball}
		Let $\X$ be a proper, $\delta$-hyperbolic metric space. Let $z\in \partial \X$, $x\in \X$ and $T\geq 0$. Then for every $r>0$ it holds that
		$\textup{Shad}_{x}\left(\xi_{x,z}\left(T\right), r\right) \subseteq B_x(z, e^{-T + r}).$
	\end{lemma}

	Using generalized visual balls one can define the classical fractal dimensions, like the Hausdorff dimension, etc. of a subset of $\partial \X$. For instance, the lower and upper Minkowski dimension of a subset $C\subseteq \partial \X$ are respectively
	$$\underline{\text{MD}}(C) := \limi_{\rho \to 0}\frac{\log \text{Cov}_x(C,\rho)}{\log \frac{1}{\rho}}, \qquad \overline{\text{MD}}(C) := \lims_{\rho \to 0}\frac{\log \text{Cov}_x(C,\rho)}{\log \frac{1}{\rho}},$$
	where $\text{Cov}_x(C,\rho)$ denotes the minimal number of generalized visual balls of radius $\rho$ needed to cover $C$. These quantities do not depend on the choice of the basepoint $x\in \X$.
	
	\subsection{Geodesic entropy of subsets of the boundary}
	\label{subsec:geodesic_entropy_subsets} 
	Let $\X$ be a proper, Gromov-hyperbolic metric space and let $C\subseteq \partial\X$. For $B \subseteq \X$, $C\subseteq \partial \X$ and $r>0$ we define
	$$\Cov_\sfd(B;C;r) := \Cov_\sfd(B \cap \text{QC-Hull}(C), r).$$
	It represents the minimal number of $r$-balls needed to cover the points of $B$ that lie inside some geodesic lines with endpoints in $C$.
	
	We introduce the adaptation of the geodesic covering entropy introduced in Section \ref{subsec:geodesic_entropy} to subsets of the boundary at infinity. The \emph{geodesic covering entropy of $\X$ with respect to $C \subseteq \partial\X$} is the quantity
	\begin{equation}
		\label{eq:defin_h_Geod}
		h_\Geod(\X;C) := \lim_{r\to 0}\lims_{T \to +\infty} \frac{1}{T} \log \Cov_\sfd(\overline{B}(x,T);C; r).
	\end{equation}
	A similar definition has been considered in \cite{Cav21}, under the name of covering entropy of $C$. In order to avoid confusion we prefer to call it geodesic covering entropy since it refers to points belonging to geodesic lines.
	
	If $C=\partial \X$ then $h_\Geod(\X;C) = h_\Geod(\X)$ as defined in Section \ref{subsec:geodesic_entropy}. 
	
	The geodesic covering entropy of $\X$ with respect to $C$ bounds from above the Minkowski dimension.
	\begin{prop}
		\label{prop:MD<h_Geod}
		Let $\X$ be a proper, geodesic, Gromov-hyperbolic metric space and let $C\subseteq \partial \X$. Then $\overline{\MD}(C) \le h_\Geod(\X;C)$.
	\end{prop}
	\begin{proof}
		If $\# C < 2$ then $\overline{\MD}(C) = 0$ and there is nothing to prove. So we can assume that $\# C \ge 2$.
		Let $\delta$ be the hyperbolicity constant of $\X$.
		Fix $x\in \X$ and let $M:=\sfd(x,\text{QC-Hull}(C))$. We will show a stronger statement that will be useful also later. Namely, we show that
		\begin{equation}
			\label{eq:MD<e_T_Geod}
			\overline{\MD}(C) \le \lim_{r\to 0}\lims_{T \to +\infty}\frac{1}{T}\log \Cov_\sfd(e_T(\Geod(\overline{B}(x,M);C)),r),
		\end{equation}
		where $\Geod(\overline{B}(x,M);C)$ denotes the set of geodesic lines $\gamma$ with both endpoints in $C$ and with $\gamma(0) \in \overline{B}(x,M)$. Since $e_T$ is the evaluation map at time $T$, every point of $e_T(\Geod(\overline{B}(x,M);C))$ belongs to $\overline{B}(x, M + T) \cap \text{QC-Hull}(C)$, so \eqref{eq:MD<e_T_Geod} implies the thesis.
		
		We now prove \eqref{eq:MD<e_T_Geod}. For every $T>M$ let $\{y_i\}$ be a $1$-dense subset of $e_T(\Geod(\overline{B}(x,M); C))$. For every $i$ let $\gamma_i \in \Geod(\overline{B}(x,M), C)$ such that $\gamma_i(T) = y_i$. Let $z\in C$. By Lemma \ref{lemma:approximation-ray-line} there exists a geodesic line $\gamma \in \Geod(\overline{B}(x,M); C)$ with $\gamma^+ = z$ and $\sfd(\xi_{x,z}(t),\gamma(t)) \le M$ for every $t\ge 0$. Then there exists $i$ such that $\sfd(\gamma(T), y_i) \le 1$. Moreover, since $T> M$ we can apply Lemma \ref{lemma:projection_quasigeodesic} to find that $\sfd(y_i, \xi_{x,\gamma_i^+}(S)) \le 144\delta$, where $S = \sfd(x,y_i)$ satisfies $\vert S - T \vert \le M$. The triangular inequality implies that $\sfd(\xi_{x,z}(T), \xi_{x,\gamma_i^+}(T)) \le 2M + 1 + 144\delta$. Then $(z,\gamma_i^+)_x > T - M - \frac{1}{2} - 72\delta$ by Lemma \ref{product-rays}. Defining $\rho := e^{-T}$ we get that $\Cov_x(C,e^{M+\frac{1}{2} + 72\delta}\rho) \le \# \{\gamma_i^+\}$. Hence,
		\begin{equation}
			\begin{aligned}
				\overline{\MD}(C) = \lims_{\rho \to 0} \frac{\log \Cov_x(C,\rho)}{\log \frac{1}{\rho}} &= \lims_{\rho \to 0} \frac{\log \Cov_x(C,e^{M+\frac12 + 72\delta}\rho)}{\log \frac{1}{\rho}} \\
				&\le \lims_{T \to +\infty} \frac{1}{T}\log \Cov_\sfd(e_T(\Geod(\overline{B}(x,M);C)),1) \\
				&\le \lim_{r\to 0} \lims_{T \to +\infty} \frac{1}{T}\log \Cov_\sfd(e_T(\Geod(\overline{B}(x,M);C)),r).
			\end{aligned}
		\end{equation}
	\end{proof}
	
	We say that $\X$ has \emph{bounded geometry with respect to $C$} if
	\begin{equation}
		\label{eq:defin_bounded_geometry_wrt_C}
		\sup_{y\in \textup{QC-Hull}(C)} \Cov_\sfd(\overline{B}(y,R);C; r) =: C_{R,r} < +\infty \quad \text{for all } 0<r \le R.
	\end{equation}

	Even when $C= \partial \X$, condition \eqref{eq:defin_bounded_geometry_wrt_C} is weaker than \eqref{eq:defin_bounded_geometry}. Indeed the supremum is only among points that are along geodesic lines, actually along geodesic lines with endpoints in $C$. Moreover, also the set to cover is not the whole ball $\overline{B}(x,R)$, but only its intersection with the set of geodesic lines with endpoints in $C$. 
	The following result is similar to Lemma \ref{lemma:bounded_geometry_no_r}.
	
	\begin{lemma}
		\label{lemma:h_erg_no_r}
		Let $\X$ be a proper, geodesic, Gromov-hyperbolic metric space, let $x\in \X$ and let $C \subseteq \partial \X$. Suppose $\X$ has bounded geometry with respect to $C$. Then the function 
		$$r\mapsto \lims_{T \to +\infty}\frac{1}{T} \log \Cov_\sfd(\overline{B}(x,T);C;r)$$
		is constant.
	\end{lemma}
	\begin{proof}
		The function $r \mapsto \lims_{T \to +\infty}\frac{1}{T} \log \Cov_\sfd(\overline{B}(x,T);C;r)$ is decreasing. For the other direction we fix $0<r\le R$ and we consider the constant $C_{R,r}$ given by \eqref{eq:defin_bounded_geometry_wrt_C}. Then 
		\begin{equation}
			\begin{aligned}
				\lims_{T \to +\infty}\frac{1}{T} \log \Cov_\sfd(\overline{B}(x,T);C;r) &\le  \lims_{T \to +\infty}\frac{1}{T} \log\left( \Cov_\sfd(\overline{B}(x,T);C,R) \cdot C_{R,r}\right)\\
				&= \lims_{T \to +\infty}\frac{1}{T} \log \Cov_\sfd(\overline{B}(x,T);C;R).
			\end{aligned}
		\end{equation}
	\end{proof}
	
	As a consequence, if \eqref{eq:defin_bounded_geometry_wrt_C} holds then we have equality in Lemma \ref{prop:MD<h_Geod}. This is a result that extends \cite[Theorem D]{Cav21}.
	\begin{prop}
		\label{prop:h_Geod=MD}
		Let $\X$ be a proper, geodesic, Gromov-hyperbolic metric space and let $C\subseteq \partial\X$ be such that $\# C \ge 2$. Suppose $\X$ has bounded geometry with respect to $C$. Then $h_\Geod(\X;C) = \overline{\MD}(C)$. 
	\end{prop}
	\begin{proof}
		Proposition \ref{prop:MD<h_Geod} gives that $\overline{\MD}(C) \le h_\Geod(\X;C)$, so we need to show the other inequality. Let $\delta$ be the hyperbolicity constant of $\X$, let $x\in \X$ and let $M:= \sfd(x,\text{QC-Hull}(C)) + 22\delta$. We claim that
		\begin{equation}
			\label{eq:h_Geod_balls=h_Geod_spheres}
			h_\Geod(\X;C) = \lims_{T \to +\infty}\frac{1}{T}\log \Cov_{\sfd}(\overline{B}(x,T);C;294\delta) = \lims_{T \to +\infty}\frac{1}{T}\log \Cov_{\sfd}(S(x,T);C;294\delta),
		\end{equation}
		where $S(x,T)$ denotes the set of points at distance exactly $T$ from $x$.
		
		Let us suppose for the moment that \eqref{eq:h_Geod_balls=h_Geod_spheres} holds and we show how to conclude the proof. Let $T > M$ be fixed and set $\rho := e^{-T}$. Let $\{z_i\} \subseteq C$ be a minimal $\rho$-dense subset of $C$, i.e. for every $z\in C$ there exists $i$ such that $(z,z_i)_x \ge T$. Lemma \ref{lemma:approximation-ray-line} gives the existence of geodesic lines $\gamma_i$ with endpoints in $C$, $\gamma_i^+ = z_i$ and $\sfd(x,\gamma_i(0))\le M$. We consider a point $y_i$ along $\gamma_i\restr{[0,+\infty)}$ such that $\sfd(x,y_i) = T$, so $y_i \in S(x,T) \cap \text{QC-Hull}(C)$. We observe that Lemma \ref{lemma:projection_quasigeodesic} implies that $\sfd(\xi_{x,z_i}(T), y_i) \le 144\delta$ for every geodesic ray $\xi_{x,z_i}$. 
		
		We claim that $\{y_i\}$ is $294\delta$-dense in $S(x,T)\cap \text{QC-Hull}(C)$.
		Let $y\in S(x,T)\cap \text{QC-Hull}(C)$, in particular it belongs to a geodesic line $\gamma$ with endpoints in $C$. We parametrize $\gamma$ in such a way that $\gamma(0)$ is a closest point of $\gamma$ to $x$ and $y\in \gamma \restr{[0,+\infty)}$. Lemma \ref{lemma:projection_quasigeodesic} implies that $\sfd(y, \xi_{x,\gamma^+}(T)) \le 144\delta$. The point $\gamma^+$ belongs to $C$, then there exists $i$ such that $(\gamma^+,z_i)_x \ge T$. Lemma \ref{product-rays} says that $\sfd(\xi_{x,z_i}(T), \xi_{x,\gamma^+}(T)) \le 6\delta$. Therefore $\sfd(y,y_i) \le 294\delta$ and the claim is proved. This implies that
		\begin{equation}
			\begin{aligned}
				\Cov_{\sfd}(S(x,T);C;294\delta) \le \#\{y_i\} = \#\{z_i\} = \Cov_x(C,\rho),
			\end{aligned}
		\end{equation}
		leading to 
		\begin{equation}
			\lims_{T \to +\infty}\frac{1}{T}\log \Cov_{\sfd}(S(x,T);C;294\delta) \le \lims_{\rho \to 0} \frac{\log \Cov_x(C,\rho)}{\log \frac{1}{\rho}} = \overline{\MD}(C).
		\end{equation}
		
		It remains to prove \eqref{eq:h_Geod_balls=h_Geod_spheres}. The first equality is Lemma \ref{lemma:h_erg_no_r}. Moreover it is obvious that $$\lims_{T \to +\infty}\frac{1}{T}\log \Cov_{\sfd}(\overline{B}(x,T);C;294\delta) \ge \lims_{T \to +\infty}\frac{1}{T}\log \Cov_{\sfd}(S(x,T);C;294\delta).$$
		So we need to prove the opposite inequality. We write $\overline{B}(x,T) \subseteq \bigcup_{k=0}^{\lfloor T \rfloor} A(x,k,k+1)$, where $A(x,k,k+1) = \overline{B}(x,k+1) \setminus B(x,k)$ and we deduce that
		$$\Cov_{\sfd}(\overline{B}(x,T);C;294\delta) \le \sum_{k=0}^{\lfloor T \rfloor} \Cov_{\sfd}(A(x,k,k+1);C;294\delta).$$
		It is therefore enough to prove that $\Cov_{\sfd}(A(x,k,k+1);C;294\delta) \le \Cov_{\sfd}(S(x,T);C;294\delta)$ for every $k=M,\ldots,\lfloor T - 291\delta - 1 \rfloor$.
		
		We fix such a $k$. Let $\{y_i\}$ be a $294\delta$-dense subset of $S(x,T)\cap \text{QC-Hull}(C)$. For every $i$ we take a geodesic line $\gamma_i$ with endpoints in $C$ containing $y_i$. We parametrize $\gamma_i$ in such a way that $\gamma_i(0)$ is a closest point of $\gamma_i$ to $x$ and $y_i$ belongs to $\gamma_i\restr{[0,+\infty)}$. We consider a geodesic ray $\xi_{x, \gamma_i^+}$ and we apply Lemma \ref{lemma:approximation-ray-line} in order to find another geodesic line $\gamma_i'$ with endpoints in $C$, $\gamma_i'^+ = \gamma_i^+$ and parametrized in such a way that $\gamma_i'(0)$ is a closest point of $\gamma_i'$ to $x$ and $\sfd(x,\gamma_i'(0))\le M$. We take a point $x_i$ that belongs to $\gamma_i'\restr{[0,+\infty)}$ and that satisfies $\sfd(x,x_i) \in [k,k+1]$. This is possible since $k \ge M$ and $\sfd(x, \gamma_i'(0)) \le M$. 
		
		We claim that the set $\{x_i\} \in \text{QC-Hull}(C) \cap A(x,k,k+1)$ is $294\delta$-dense. Let $y \in \text{QC-Hull}(C) \cap A(x,k,k+1)$, let $\gamma$ be a geodesic line with endpoints in $C$ containing $y$ and parametrized in such a way that $\gamma(0)$ is a closest point of $\gamma$ to $x$ and $x\in \gamma\restr{[0,+\infty)}$. Let $y_T$ be a point of $\gamma\restr{[0,+\infty)}$ belonging to $S(x,T)$. Then there exists $i$ such that $\sfd(y_T,y_i) \le 294\delta$. We consider geodesic rays $\xi_{x,\gamma^+}$ and $\xi_{x,\gamma_i^+}$. Lemma \ref{lemma:projection_quasigeodesic} implies that $\sfd(\xi_{x,\gamma^+}(T), \xi_{x,\gamma_i^+}(T)) \le 582\delta$. This means, by Lemma \ref{product-rays}, that $(\gamma_i^+, \gamma^+)_x \ge T - 291\delta$. Therefore $(\gamma_i'^+, \gamma^+)_x \ge T - 291\delta$, since $\gamma_i'^+ = \gamma_i^+$. Again by Lemma \ref{product-rays} we get that $\sfd(\xi_{x,\gamma^+}(t), \xi_{x,\gamma_i'^+}(t)) \le 6\delta$ for every $t \in [0,T-291\delta]$. Another double application of Lemma \ref{lemma:projection_quasigeodesic} gives that $\sfd(x_i, y) \le 294\delta$, where we used the fact that both $\sfd(x,x_i)$ and $\sfd(x,y)$ are at most $k+1 \le T - 291\delta$.	This concludes the proof.
	\end{proof}
	
	We end this section with another property implied by the bounded geometry assumption.
	\begin{prop}
		\label{prop:bounded_geometry_implies_doubling_boundary}
		Let $\X$ be a proper, geodesic, Gromov-hyperbolic metric space, let $x\in \X$ and let $C\subseteq \partial \X$. Suppose $\X$ has bounded geometry with respect to $C$. Then $C$ is doubling in the following sense: there exists $C_D > 0$ such that
		\begin{equation}
			\label{eq:covering_doubling_boundary}
			\Cov_x\left(B_x(z,\rho), \frac{\rho}{2}\right) \le C_D
		\end{equation}
		for every $z\in C$, every $\rho > 0$. In particular, $\overline{\MD}(C) < +\infty$.
	\end{prop}
	\begin{proof}
		Let $\delta$ be the hyperbolicity constant of $\X$. Let $z\in C$ and $\rho > 0$. Set $T:= \log\frac{1}{\rho}$.
		Let $S:= 72\delta + 1 + \log 2$ and let $R:= S + 6\delta$. Let $\{x_i\}$ be a $1$-dense subset of $\overline{B}(\xi_{x,z}(T), R) \cap \text{QC-Hull}(C)$ of cardinality at most $C_D := C_{R,1}$,  where the right hand side comes directly from \eqref{eq:defin_bounded_geometry_wrt_C}. For every $i$ let $\gamma_i$ be a geodesic line with endpoints in $C$ and parametrized in such a way that $\gamma_i(0)$ is a closest point of $\gamma_i$ to $x$ and $x_i \in \gamma_i \restr{[0,+\infty)}$. We claim that the set $\{\gamma_i^+\}$ is $\frac{\rho}{2}$-dense in $B_x(z,\rho)$, i.e. for every $w \in B_x(z,\rho)$ there exists $i$ such that $(w,\gamma_i^+)_x > \log \frac{2}{\rho} = T + \log 2$. Let $w \in B_x(z,\rho)$. By Lemma \ref{product-rays} we have that $\sfd(\xi_{x,z}(T), \xi_{x,w}(T)) \le 6\delta$, so the point $\xi_{x,w}(T+S)$ belongs to $\overline{B}(\xi_{x,z}(T),R)$. Therefore there exists $i$ such that $\sfd(\xi_{x,w}(T+S), x_i) \le 1$. Lemma \ref{lemma:projection_quasigeodesic} implies that $\sfd(x_i, \xi_{x,\gamma_i^+}(t)) \le 144\delta$ where $t$ is such that $\sfd(x, \xi_{x,\gamma_i^+}(t)) = \sfd(x, x_i)$. Since $\sfd(\xi_{x,w}(T+S), x_i) \le 1$ we obtain that $\sfd(x_i, \xi_{x,\gamma_i^+}(T + S)) \le 144\delta + 1$. This means that $\sfd(\xi_{x,\gamma_i^+}(T + S), \xi_{x,w}(T + S)) \le 144\delta + 2$. Lemma \ref{product-rays} says that $(\gamma_i^+,w)_x > T + S - 72\delta - 1 = T + \log 2 = \log \frac{2}{\rho}$, i.e. $w\in B_x(\gamma_i^+, \frac{\rho}{2})$.
		
		For the second part we argue as follows. The diameter of $\partial \X$ is at most $1$, so it is for $C$. In other words, $C \subseteq B(z, 1)$ for every $z \in C$. For every $k > 0$, by iterating \eqref{eq:covering_doubling_boundary}, we obtain that $\Cov_x(C, \frac{1}{2^k}) \le C_D^k$. Similarly, for every $\rho > 0$ we have that $\Cov_x(C, \rho) \le C_D^{\log_2 \frac{1}{\rho}}$.
		This implies that
		$$\overline{\MD}(C) \le \lims_{\rho\to 0} \frac{\log_2 \frac{1}{\rho} \cdot \log C_D}{\log \frac{1}{\rho}} = \frac{\log C_D}{\log 2} < +\infty.$$
	\end{proof}
	
	\subsection{Groups of isometries and Bishop-Jones Theorem}
	\label{subsec:BJ}
	If $\X$ is a proper metric space we denote its group of isometries by $\Isom(\X)$ and we endow it with the uniform convergence on compact subsets of $\X$. A subgroup $\Gamma$ of $\Isom(\X)$ is called {\em discrete} if the following equivalent conditions hold:
	\begin{itemize}
		\item[(a)] $\Gamma$ is discrete as a subspace of $\Isom(\X)$;
		\item[(b)] for all $x\in \X$ and $R\geq 0$ the set $\Sigma_R(x) = \lbrace g \in \Gamma  \,:\,  g  x\in \overline{B}(x,R)\rbrace$ is finite.
	\end{itemize}   
	
	The critical exponent of a discrete group of isometries $\Gamma$ acting on a proper metric space $\X$ can be defined using the Poincaré series, or alternatively (\cite{Cav21ter}, \cite{Coo93}), as
	\begin{equation}
		h_\text{crit}(\X;\Gamma) = \lims_{T \to +\infty}\frac{1}{T}\log \# (\Gamma x \cap B(x,T)),
	\end{equation}
	where $x$ is a fixed point of $\X$. This quantity does not depend on the choice of $x$.

	We specialize the situation to the case of a proper, Gromov hyperbolic metric space $\X$. Every isometry of $\X$ acts naturally on $\partial \X$ and the resulting map on $\X\cup \partial \X$ is a homeomorphism.
	The {\em limit set} $\Lambda(\Gamma)$ of a discrete group of isometries $\Gamma$ is the set of accumulation points of the orbit $\Gamma x$ on $\partial \X$, where $x$ is any point of $\X$; it is the smallest $\Gamma$-invariant closed set of the Gromov boundary (cp. \cite[Theorem 5.1]{Coo93}) and it does not depend on $x$. The group $\Gamma$ is said to be \emph{non-elementary} if $\# \Lambda(\Gamma) > 2$. If $\X$ is Gromov-hyperbolic and $\Gamma$ is non-elementary then the limit superior in \eqref{eq:defin_critical_exponent} is a true limit, see \cite[Theorem B]{Cav24}.
	
	There are several interesting subsets of the limit set like the radial limit set, the uniformly radial limit set and the ergodic limit set. The last one has been defined in \cite{Cav24}. They are related to important sets of the quotient geodesic flow of $\Gamma \backslash \X$.
	
	In order to recall their definition, we fix a basepoint $x\in \X$, a positive real number $\tau > 0$ and an increasing sequence of real numbers $\Theta = \{\vartheta_i\}_{i\in \N}$ such that $\lim_{i\to +\infty} \vartheta_i = +\infty$. 	We define
	$\Lambda_{\tau, \Theta}(\Gamma)$ as the set of points $z\in \partial \X$ for which there exists a geodesic ray $\xi_{x,z}$ satisfying the following property: for every $i\in \N$ there exists $t_i\in [\vartheta_i,\vartheta_{i+1}]$ such that $\sfd(\xi_{x,z}(t_i), \Gamma x) \le \tau$.
	
	By choosing different classes of sequences $\Theta$, we define the different sets we recalled above. The \emph{radial limit set} is
	$$\Lambda_\text{rad}(\Gamma) := \bigcup_{\tau > 0} \bigcup_{\Theta} \Lambda_{\tau, \Theta},$$
	where the second union is taken among all possible sequences $\Theta$ as above. If we restrict the class of sequences we consider, we define the other sets. For $\tau > 0$ we define $\Lambda_\tau(\Gamma) := \Lambda_{\tau, \{i\tau\}}(\Gamma)$. The \emph{uniform radial limit set} of $\Gamma$ is
	\begin{equation}
		\label{eq:defin-uniform-radial-set}
		\Lambda_{\text{u-rad}}(\Gamma) := \bigcup_{\tau > 0} \Lambda_\tau(\Gamma).
	\end{equation}
	We define $\Theta_\text{erg}$ as the set of increasing sequences of positive real numbers $\Theta = \{\vartheta_i\}_{i\in\N}$ such that $\exists \lim_{i\to +\infty} \frac{\vartheta_i}{i} \in (0,+\infty)$. The \emph{ergodic limit set} of $\Gamma$ is the set 
	\begin{equation}
		\label{eq:defin-ergodic-limit-set}
		\Lambda_{\text{erg}}(\Gamma) := \bigcup_{\tau > 0} \bigcup_{\Theta \in \Theta_\text{erg}}  \Lambda_{\tau, \Theta}(\Gamma).
	\end{equation}

	The ergodic limit set plays an important role in relation with the quotient geodesic flow of $\Gamma \backslash \X$, as \cite[Theorem C]{Cav24} shows. In Section \ref{sec:proof_of_h_top_<h_Gamma} we will expand this relation. 
	
	In the sequel we will need other subsets. Let $c\in (0,+\infty)$, $\varepsilon > 0$ and $n\in \N$. We define $$\Theta_{c,\varepsilon,n} := \left\{\{\vartheta_i\}_{i\in \N} \text{ increasing sequence of real numbers}\,:\, \left \vert \frac{\vartheta_i}{i} - c \right \vert < \varepsilon\text{ for all } i\ge n \right\}$$
	and
	$$\Lambda_{\tau,c, \varepsilon, n}(\Gamma) :=  \bigcup_{\Theta \in \Theta_{c,\varepsilon,n}} \Lambda_{\tau, \Theta}(\Gamma).$$
	Notice that a sequence $\{\vartheta_i\} \in \Theta_{c,\varepsilon,n}$ satisfies $c- \varepsilon \le \limi_{i\to +\infty} \frac{\vartheta_i}{i} \le  \lims_{i\to +\infty} \frac{\vartheta_i}{i} \le c+\varepsilon$.

	When $\Gamma$ is clear from the context, we simply omit it from the notation and we write $\Lambda, \Lambda_\text{rad}, \Lambda_\text{u-rad}, \Lambda_\text{erg}, \Lambda_{\tau,c,\varepsilon,n}$. The sets $\Lambda, \Lambda_\text{rad}, \Lambda_\text{u-rad}$ and $\Lambda_\text{erg}$ are $\Gamma$-invariant (\cite[Lemma 4.2]{Cav24}), while the sets $\Lambda_{\tau,c,\varepsilon,n}$ are not in general.

	In the proof of Theorem \ref{theo:intro-h_top>=h_Gamma} we will use following version of Bishop-Jones' Theorem that has been proved in  the remark after the proof of \cite[Theorem B]{Cav24}.
	\begin{theo}[\cite{BJ97, Pau97, DSU17, Cav24}]
		\label{theo:Bishop-Jones}
		Let $\X$ be a proper, Gromov-hyperbolic metric space and let $\Gamma < \Isom(\X)$ be discrete and non-elementary. Then
		\begin{equation}
			\begin{aligned}
				h_\textup{crit}(\X;\Gamma) =  \sup_{\tau \geq 0} \underline{\textup{MD}}(\Lambda_{\tau}).
			\end{aligned}
		\end{equation}
	\end{theo}

	One of the key steps in the proof of Theorem \ref{theo:intro-equality-Gromov-hyperbolic} will be to show that every ergodic measure of the quotient geodesic flow of $\Gamma \backslash \X$ is concentrated on the set of geodesics with endpoints in $\Lambda_{\tau,c, \varepsilon, n}$ for some choices of the parameters. 
	For the proof of one inequality in Theorem \ref{theo:intro-equality-Gromov-hyperbolic} we will need
	the next estimate, whose proof follows the same lines of \cite[Theorem B]{Cav24}.
	
	\begin{prop}
		\label{prop:MD < h_Gamma_special_subsets}
		Let $\X$ be a proper, Gromov-hyperbolic metric space, let $x\in \X$ and let $\Gamma < \Isom(\X)$ be discrete. For every $\tau,c,\varepsilon > 0$ and $n\in \N$ it holds that
		$$\overline{\MD}(\Lambda_{\tau,c, \varepsilon, n}) \le h_\textup{crit}(\X;\Gamma) \cdot \frac{c+\varepsilon}{c-\varepsilon}.$$
	\end{prop}
	\begin{proof}
		Let $\delta$ be the hyperbolicity constant of $\X$.
		For simplicity, we set $h := h_\text{crit}(\X;\Gamma)$.
		We fix an error term $\eta > 0$.	Let $T_\eta > 0$ be such that 
		$$\frac{1}{T}\log \#(\Gamma x \cap B(x,T)) \leq h + \eta$$
		for every $T\geq T_\eta$.
		Let $\rho \leq e^{-\max\lbrace T_\eta, n(c-\varepsilon) \rbrace}$. We consider $j_\rho\in \mathbb{N}$ with the following property: $(j_\rho-1)(c - \varepsilon) < \log \frac{1}{\rho}\leq j_\rho(c - \varepsilon)$. The condition on $\rho$ gives $\log \frac{1}{\rho} \geq n(c - \varepsilon)$, implying $j_\rho\geq n$.\\
		We consider the set of elements $g\in \Gamma$ such that 
		\begin{equation}
			\label{shadow}
			j_\rho(c - \varepsilon) - \tau \leq \sfd(x,gx) \leq (j_\rho+1)(c + \varepsilon) + \tau.
		\end{equation}
		For every such $g$ we consider the shadow $\text{Shad}_{x}(gx,2\tau + 8\delta)$ and we claim that this set of shadows covers $\Lambda_{\tau,c, \varepsilon, n}$. 
		
		Let $z \in \Lambda_{\tau,c, \varepsilon, n}$. By definition, $z\in \Lambda_{\tau,\{\vartheta_i\}}$ with the property that for every $i\ge n$ it holds that $(c-\varepsilon)i \le \vartheta_i \le (c+\varepsilon)i$. In particular this holds for $j_\rho$, i.e. $(c-\varepsilon)j_\rho \le \vartheta_{j_\rho} \le (c+\varepsilon)j_\rho$.
		Hence there exists a geodesic ray $\xi_{x,z}$ and a time $t$ satisfying
		$$j_\rho(c - \varepsilon)\leq \vartheta_{j_\rho} \leq t\leq \vartheta_{j_\rho+1}\leq (j_\rho+1)(c + \varepsilon), \qquad d(\xi_{x,z}(t),\Gamma x) \leq \tau.$$
		So there exists $g\in \Gamma$ satisfying \eqref{shadow} such that $z\in \text{Shad}_{x}(gx,2\tau + 8\delta)$, by Lemma \ref{parallel-geodesics}.
		Moreover, all these shadows are casted by points at distance at least $j_\rho(c - \varepsilon) - \tau$ from $x$, so at distance at least $\log\frac{1}{e^\tau\rho}$ from $x$. We need to estimate the number of such $g$'s.
		By the assumption on $\rho$ we get that this number is less than or equal to $e^{(h+\eta)[(j_\rho+1)(c + \varepsilon) + \tau]}.$
		Hence, by Lemma \ref{shadow-ball}, we conclude that $\Lambda_{\tau,c, \varepsilon, n}$ is covered by at most $e^{(h+\eta)[(j_\rho+1)(c + \varepsilon) + \tau]}$ generalized visual balls of radius $e^{3\tau+8\delta}\rho$. Thus,
		\begin{equation*}
			\begin{aligned}
				\overline{\text{MD}}(\Lambda_{\tau,c, \varepsilon, n})&=\lims_{\rho \to 0}\frac{\log\text{Cov}_x(\Lambda_{\tau,c, \varepsilon, n}, e^{3\tau+8\delta}\rho)}{\log \frac{1}{e^{3\tau+8\delta}\rho}}\\
				&\leq \lims_{\rho\to 0}\frac{(h+\eta)[(j_\rho+1)(c + \varepsilon) + \tau]}{-3\tau -8\delta + (j_\rho-1)(c - \varepsilon)} \\
				&\leq (h+\eta)\cdot \frac{c+\varepsilon}{c-\varepsilon}.
			\end{aligned}
		\end{equation*}
		The thesis follows by the arbitrariness of $\eta > 0$.
	\end{proof}

\subsection{Ergodic entropy of a group of isometries}
\label{subsec:ergodic_entropy} 	
	Let $\X$ be a geodesic, proper, Gromov-hyperbolic metric space. Let $\Gamma < \Isom(\X)$ be discrete. We define the \emph{ergodic entropy} of $\Gamma$ as
	$$h_\textup{erg}(\X;\Gamma) := h_\Geod(\X; \Lambda_{\text{erg}}(\Gamma)).$$
	This is the notion that appears in Theorem \ref{theo:intro-equality-Gromov-hyperbolic}. We observe the following property.
	\begin{cor}
		\label{cor:h_crit<h_erg}
		Let $\X$ be a proper, geodesic, Gromov-hyperbolic metric space and let $\Gamma < \Isom(\X)$ be discrete and non-elementary. Then $h_\textup{crit}(\X;\Gamma) \le h_\textup{erg}(\X;\Gamma)$. 
	\end{cor}
	\begin{proof}
		By Theorem \ref{theo:Bishop-Jones} we have that $h_\textup{crit}(\X;\Gamma) = \sup_{\tau \ge 0} \underline{\MD}(\Lambda_\tau)$. Proposition \ref{prop:MD<h_Geod} implies that $\underline{\MD}(\Lambda_\tau) \le h_\Geod(\X;\Lambda_\tau)$ for every $\tau \ge 0$ big enough. The thesis follows since for every such $\tau$ we have that $\Lambda_\tau \subseteq \Lambda_\text{erg}$, so $\underline{\MD}(\Lambda_\tau) \le h_\Geod(\X;\Lambda_\text{erg})$.
	\end{proof}
	
	In particular, under the hypotheses of Theorem \ref{theo:intro-equality-Gromov-hyperbolic} we have that $h_\text{crit}(\X;\Gamma) < +\infty$.
	
	\begin{obs}
		\label{rmk:h_Gamma_vs_h_erg}
		In this remark we explain one difference between the conditions $h_\textup{crit}(\X;\Gamma) < +\infty$ and $h_\textup{erg}(\X;\Gamma) < +\infty$. In the definition of $h_\textup{erg}(\X;\Gamma)$ we take into account the covering at every scale $r > 0$ and then we consider the limit for $r$ going to zero, while in the definition of $h_\textup{crit}(\X;\Gamma)$ we just count the spread of the orbit, without looking at a particular scale. When we want to define quantities related to the topological entropy of the geodesic flow we really need to consider every possible scale $r>0$. Indeed in the classical definition of topological entropy that we will see in Section \ref{sec:geodesic_flow}, every scale $r>0$ is considered. The problem of the scales will be central in the proof of Theorem \ref{theo:intro-equality-Gromov-hyperbolic}, and in Section \ref{sec:proof_of_h_top_<h_Gamma} we will see how the condition $h_\textup{erg}(\X;\Gamma) < +\infty$, together with the line-convexity, gives a solution to this problem. However, the dependence of $h_\textup{erg}(\X;\Gamma)$ on every scale $r>0$ is quite weak as the next corollary explains.
	\end{obs}

	\begin{cor}
		\label{cor:bounded_geometry_then_finite_erg_entropy}
		Let $\X$ be a proper, geodesic, Gromov-hyperbolic metric space, let $x\in \X$ and let $\Gamma < \Isom(\X)$ be discrete and non-elementary. Suppose $\X$ has bounded geometry with respect to $\Lambda_\textup{erg}$. Then
		\begin{itemize}
			\item[(i)] the function $r\mapsto \lims_{T \to +\infty}\frac{1}{T} \log \Cov_\sfd(\overline{B}(x,T);\Lambda_\textup{erg};r)$ is constant;
			\item[(ii)] $h_\textup{erg}(\X;\Gamma) = \overline{\MD}(\Lambda) < +\infty$.
		\end{itemize}	
	\end{cor}
	\begin{proof}
		The first statement follows immediately from Lemma \ref{lemma:h_erg_no_r}. Moreover, Proposition \ref{prop:h_Geod=MD} implies that $h_\text{erg}(\X;\Gamma) = \overline{\MD}(\Lambda_\text{erg}) = \overline{\MD}(\Lambda)$, where the last equality comes from the fact that $\Lambda_\text{erg}$ is dense in $\Lambda$, being $\Gamma$-invariant. The finiteness of $h_\text{erg}(\X;\Gamma)$ is then consequence of Proposition \ref{prop:bounded_geometry_implies_doubling_boundary}.
	\end{proof}

	\begin{obs}
		\label{rmk:difference_h_erg_h_crit}
		We can now see a second difference between $h_\textup{erg}(\X;\Gamma)$ and $h_\textup{crit}(\X;\Gamma)$. If $\X$ is a proper, geodesic, Gromov-hyperbolic metric space with bounded geometry with respect to $\Lambda_\textup{erg}$ then $h_\textup{erg}(\X;\Gamma)$ coincides with the Minkowski dimension of the limit set $\Lambda$, or equivalently of the uniformly radial limit set $\Lambda_\textup{u-rad}$, by density. On the other hand the Bishop-Jones' Theorem in the version of \cite{DSU17} says that $h_\textup{crit}(\X;\Gamma)$ coincides with the Hausdorff dimension of $\Lambda_\textup{u-rad}$. Therefore the difference between the two quantities, at least under these assumptions, relies on the fractal properties of the set $\Lambda_\textup{u-rad}$.
		
		Additionally we remark that there exist manifold examples where the two quantities above differ. For instance in \cite{PS19} there is an example of a simply connected manifold $\tilde{M}$ with pinched negative sectional curvature admitting a finite volume, discrete, torsion-free group of isometries $\Gamma$, but such that $h_\textup{crit}(\tilde{M};\Gamma)$ is strictly smaller than the volume entropy of $\tilde{M}$. In this case $\Lambda(\Gamma) = \partial \tilde{M}$, so \cite[Theorem B]{Cav21} implies that the volume entropy of $\tilde{M}$ coincides with $\overline{\MD}(\Lambda)$.
	\end{obs}

	For quasiconvex cocompact groups, i.e. groups $\Gamma$ such that the quotient space $\Gamma \backslash \text{QC-Hull}(\Lambda(\Gamma))$ is compact, we have equality between the ergodic entropy and the critical exponent.

	\begin{cor}
		\label{cor:h_crit=h_erg}
		Let $\X$ be a proper, geodesic, Gromov-hyperbolic metric space and let $\Gamma < \Isom(\X)$ be discrete, non-elementary and quasiconvex cocompact. Then it holds that $h_\textup{erg}(\X;\Gamma) = h_\textup{crit}(\X;\Gamma) < +\infty$.
	\end{cor}
	\begin{proof}
		By quasiconvex cocompactness we have that $h_\text{crit}(\X;\Gamma) < +\infty$ by \cite[Proposition 5.7]{Cav21ter} and that $\X$ has bounded geometry with respect to $\Lambda$ by \cite[Lemma 5.4]{Cav21ter}.
		By \cite[Theorem 12.2.7]{DSU17} we have that $\Lambda = \Lambda_\textup{u-rad}$. 
		The conclusion follows by the discussion in Remark \ref{rmk:difference_h_erg_h_crit}, since $\Lambda$ is Ahlfors regular (cp.\cite{Coo93}, \cite{Cav21ter}) and so its Hausdorff and Minkowski dimensions coincide.
	\end{proof}
	
	\section{Geodesic flow}
	\label{sec:geodesic_flow}
	
	In this section we introduce the notion of geodesic flow on the quotient of a metric space by a group action. Before doing that we need to state some basic properties of the quotient metric space.
	
	\subsection{Quotient metric space}
	\label{subsec:quotient-metric-space}
	
	Let $\X$ be a proper metric space and $\Gamma < \Isom(\X)$ be discrete. We denote by $\pi\colon \X \to \Gamma \backslash \X$ the standard projection on the quotient space. On $\Gamma \backslash \X$ it is defined the standard pseudometric $\bar{\sfd}([x],[y]) := \inf_{g\in \Gamma}\sfd(x, gy)$. Since the action is discrete then this pseudometric is actually a metric. 
	The map $\pi$ is $1$-Lipschitz.
	
	The \emph{systole of $\Gamma$ at a point $x\in \X$} is defined by $\text{sys}(\Gamma,x) := \inf_{g\in \Gamma \setminus \{\text{id}\}} \sfd(x,gx)$. The map $x \mapsto \text{sys}(\Gamma,x)$ is lower semicontinuous and $\Gamma$-invariant, so it defines a lower semicontinuous function on the metric space $\Gamma \backslash \X$. Moreover, if $\Gamma$ acts freely, i.e. if $gx = x$ for some $x\in \X$ implies that $g=\text{id}$, then $\text{sys}(\Gamma,x) > 0$ for every $x\in \X$. As a consequence we have the following useful fact.
	\begin{lemma}
	\label{lemma:systole_positive}
		Let $\X$ be a proper metric space and let $\Gamma < \Isom(\X)$ be discrete and free. For every compact set $K\subseteq \Gamma \backslash \X$, there exists a positive constant $\sigma > 0$ such that the systole of $\Gamma$ at every point of $K$ is bigger than or equal to $\sigma$.
	\end{lemma}
The following lemma is standard.
\begin{lemma}
	\label{lemma-inj-radius}
	Let $\X$ be a proper metric space and let $\Gamma < \Isom(\X)$ be discrete and free. Then
	\begin{itemize}
		\item[(i)] for all $x,y\in \X$ such that $\bar{\sfd}([x], [y]) < \frac{\textup{sys}(\Gamma,x)}{2}$ there exists a unique $g\in \Gamma$ such that $\sfd(x,gy)<\frac{\textup{sys}(\Gamma,x)}{2}$. In particular $\sfd(x,gy)=\bar{\sfd}([x],[y])$;
		\item[(ii)] if $\X$ is a length space then $\pi$ is a locally isometric covering.
	\end{itemize}
\end{lemma}
	
%
	\subsection{The space of local geodesics}
	\label{subsec:space_local_geodesics}
	
	Let $\X$ be a proper, geodesic metric space and let $\LocGeod{\X}$ be its space of globally defined local geodesics. We equip it with the topology of uniform convergence on compact subsets of $\R$. This is a metrizable topology. We use metrics that are slightly different to the ones introduced in \cite{Cav21}. Let $\mathcal{F}$ be the class of continuous functions $f\colon \mathbb{R} \to (0,1]$ satisfying
	\begin{itemize}
		\item[(a)] $f(0)=1$;
		\item[(b)] $f(s) = f(-s)$ for every $s \in \mathbb{R}$;
		\item[(c)]  $\lims_{s \to +\infty} 2\vert s \vert f(s) = 0$.
	\end{itemize}
	The class $\mathcal{F}$ is not empty, as it contains for instance the functions of the form $f_a(s) := e^{-a\vert s \vert}$, for every $a > 0$. For every $f\in \mathcal{F}$, we define the quantity 
	\begin{equation}
		\label{f-distance}
		\sfD_f(\gamma, \gamma') := \sup_{s\in \R}\sfd(\gamma(s),\gamma'(s))f(s),
	\end{equation}
	where $\gamma,\gamma' \in \LocGeod{\X}$.
	\begin{lemma}
		\label{lemma-metric-locgeod}
		Let $\X$ be a proper metric space. Then
		\begin{itemize}
			\item[(i)] for every $\gamma \in \LocGeod{\X}$ and for every $t<s$ it holds $\sfd(\gamma(t),\gamma(s))\leq \ell(\gamma\vert_{[t,s]})= \vert s -t \vert$, where $\ell(\cdot)$ denotes the length of a curve;
			\item[(ii)] for every $\gamma, \gamma' \in \LocGeod{\X}$ and for every $f\in \mathcal{F}$ it holds that
			$$\sfD_f(\gamma, \gamma') \leq \sfd(\gamma(0),\gamma'(0)) + \sup_{s\in \R}2\vert s \vert f(s).$$
			Moreover, $\sfD_f(\cdot,\cdot)$ defines a distance on $\LocGeod{\X}$ which induces its topology. Finally,
			$\sfd(\gamma(0),\gamma'(0)) \leq \sfD_f(\gamma, \gamma').$
		\end{itemize}
	\end{lemma}
	\begin{proof}
		The proof of (i) is well known. The first inequality in (ii) follows directly from (i). The fact that the right hand side is finite follows by property (c) above. The last inequality in (ii) follows by property (a), while $\sfD_f$ is a distance since $\sfd$ is. If $\gamma_j$ converges uniformly to $\gamma_\infty$ on compact subsets of $\R$ then $\sfD_f(\gamma_j,\gamma_\infty) \to 0$, again because of property (c) above. Vice versa, if $\sfD_f(\gamma_j,\gamma_\infty) \to 0$ then $\gamma_j$ converges uniformly on compact subsets of $\R$ to $\gamma_\infty$ because $f$ is positive and continuous.
	\end{proof}

	The set of geodesic lines, $\Geod(\X)$, is a closed subset of $\LocGeod{\X}$. For Busemann convex metric space, so for instance in the $\CAT(-1)$ case, $\Geod(\X) = \LocGeod{\X}$.
	\begin{ex}
		\label{ex:LocGeod-not-loc-compact}
		The space $\LocGeod{\X}$ is not necessarily locally compact. Consider for instance the space $\X$ obtained by gluing infinitely many circles $\X_i$ of radius $\frac{1}{i}$ by one common point, endowed with the length metric. $\X$ is compact, but there are sequences of local geodesics that converge to the map constantly equal to the gluing point. 
		
		On the other hand the space $\Geod(\X)$ is locally compact and even proper, with the metric induced by $f\in \mathcal{F}$. This is because of Lemma \ref{lemma-metric-locgeod} and Ascoli-Arzelà's Theorem.
	\end{ex}
	In view of Example \ref{ex:LocGeod-not-loc-compact}, we mainly focus on the space $\Geod(\X)$ instead than on $\LocGeod{\X}$, even if many of the things we are going to say apply also to $\LocGeod{\X}$.
	
	Every isometry $g$ of $\X$ acts on $\Geod(\X)$ by $(g\gamma)(\cdot) = g\gamma(\cdot)$,  $\gamma\in \Geod(\X)$. The action is continuous, so $g$ defines a homeomorphism of $\Geod(\X)$. Thus every group of isometries of $\X$ acts by homeomorphisms on $\Geod(\X)$. The same action can be naturally defined on $\LocGeod{\X}$.
	
	In the next result we show that also $\Gamma \backslash \Geod(\X)$ is metrizable and proper, for every discrete group $\Gamma < \Isom(\X)$. 
	\begin{prop}
		\label{prop:quotient_of_loc_geod}
		Let $\X$ be a proper metric space and let $\Gamma < \Isom(\X)$ be discrete. For every $f\in \mathcal{F}$, the action of $\Gamma$ on $(\Geod(\X), \sfD_f)$ is by isometries and $\Gamma$, as a subgroup of $\Isom(\Geod(\X),\sfD_f)$, is discrete. In particular, the space $\Gamma \backslash \Geod(\X)$ is metrizable and proper. 
		If, moreover, $\Gamma$ acts freely on $\X$, then it acts freely also on $\Geod(\X)$.  
	\end{prop}
	\begin{proof}
		The fact that the action of $\Gamma$ on $(\Geod(\X),\sfD_f)$ is by isometries follows directly from \eqref{f-distance}. Lemma \ref{lemma-metric-locgeod} implies that for every $\gamma \in \Geod(\X)$ and every $R\geq 0$ it holds that
		$$\#\lbrace g \in \Gamma \,:\, \sfD_f(g\gamma, \gamma) \leq R\rbrace \leq \#\lbrace g \in \Gamma \,:\, \sfd(g\gamma(0), \gamma(0)) \leq R\rbrace < +\infty.$$ 
		This shows that $\Gamma$ is discrete as subgroup of isometries of $(\Geod(\X),\sfD_f)$. The discussion in Section \ref{subsec:quotient-metric-space} implies that the space $\Gamma \backslash \Geod(\X)$ is metrizable via the quotient metric $\bar{\sfD}_f$ induced by $\sfD_f$. Let $[\gamma] \in \Gamma \backslash \Geod(\X)$ and let $[\gamma_k] \in \Gamma \backslash \Geod(\X)$ be a sequence such that $\bar{\sfD}_f([\gamma_k], [\gamma]) \le R$ for every $k$. By definition of quotient metric, for every $k$ there exists $g_k\in \Gamma$ such that $\sfD_f(g_k\gamma_k, \gamma) \le R$. Lemma \ref{lemma-metric-locgeod} and Ascoli-Arzelà's Theorem imply that there exists $\gamma_\infty \in \Geod(\X)$ and a subsequence $\{k_j\}$ such that $g_{k_j}\gamma_{k_j}$ converges to $\gamma_\infty$. By definition, this means that $[\gamma_{k_j}]$ converges to $[\gamma_\infty]$, so the $\bar{\sfD}_f$ balls in $\Gamma \backslash \Geod(\X)$ are compact.
		
		Finally, suppose that $\Gamma$ acts freely on $\X$ and let $g\in \Gamma$ be such that $g\gamma = \gamma$ for some $\gamma \in \Geod(\X)$. By definition, $g\gamma(0) = \gamma(0)$ and so $g = \text{id}$.
	\end{proof}

	Under some additional assumptions, the space $\Gamma \backslash \Geod(\X)$ can be identified with the space of all local geodesics of the quotient space $\Gamma \backslash \X$. We first need a result about the space of local geodesics.

	\begin{prop}
		\label{prop-quotient-geodesics}
		Let $\X$ be a proper, geodesic metric space and let $\Gamma < \Isom(\X)$ be discrete and free. 		
		Then the map 
		$$\Pi\colon \LocGeod{\X}\mapsto  \LocGeod{\Gamma \backslash \X}, \quad \Pi(\gamma)(\cdot) = (\pi \circ \gamma)(\cdot),$$ 
		where $\pi\colon \X \to \Gamma \backslash \X$ is the projection map,
		is continuous, surjective and $\Gamma$-invariant. So it induces a map  $$\bar\Pi\colon \Gamma\backslash\LocGeod{\X}\mapsto  \LocGeod{\Gamma \backslash \X}$$ 
		which is a homeomorphism.
	\end{prop}
	\begin{proof}
		First of all, $\Pi$ is well defined. Indeed the projection of a local geodesic of $\X$ on $\Gamma \backslash \X$ is a local geodesic, by Lemma \ref{lemma-inj-radius}.
		Moreover, $\Pi$ is $\Gamma$-invariant: if $\gamma = g \gamma'$ for some $g\in \Gamma$ and $\gamma,\gamma' \in \LocGeod{\X}$, then $(\pi \circ \gamma)(\cdot) = (\pi\circ\gamma')(\cdot)$.
		Furthermore, every map $\gamma\colon \mathbb{R} \to \Gamma \backslash \X$ can be lifted to a map $\tilde{\gamma}\colon \mathbb{R}\to\ X$ since $\pi\colon \X\to \Gamma\backslash \X$ is a covering map by Lemma \ref{lemma-inj-radius}. If $\gamma$ is a local geodesic of $\Gamma \backslash \X$ then $\tilde{\gamma}$ is a local geodesic of $\X$, again by Lemma \ref{lemma-inj-radius}. This shows that $\Pi$ is surjective. Since $\pi$ is $1$-Lipschitz when we equip both $\LocGeod{\X}$ and $\LocGeod{\Gamma \backslash \X}$ with the metric induced by the same $f\in \mathcal{F}$, then $\Pi$ is continuous.
		
		It remains to show that the map $\bar{\Pi}$ is a homeomorphism. Let us show it is injective: suppose $\Pi(\gamma) = \Pi(\gamma')$, i.e. $(\pi\circ\gamma)(t)=(\pi\circ\gamma')(t)$ for every $t\in \mathbb{R}$. For every $t\in \mathbb{R}$ there is a unique $g_t\in \Gamma$ such that $g_t\gamma'(t) = \gamma(t)$, by Lemma \ref{lemma-inj-radius}.
		We consider the set $I=\lbrace t \in \mathbb{R} \,:\, g_t = g_0\rbrace$. It is closed: indeed if $t_k\to t_\infty$ with $t_k\in I$ then we have
		\begin{equation*}
			\begin{aligned}
				\sfd(g_{t_\infty}\gamma'(t_\infty), g_0\gamma'(t_\infty)) &\leq \sfd(g_{t_\infty}\gamma'(t_\infty), g_{t_k}\gamma'(t_k)) + \sfd(g_{t_k}\gamma'(t_k), g_0\gamma'(t_\infty))\\
				&\leq \sfd(\gamma(t_\infty),\gamma(t_k)) + \sfd(\gamma'(t_k),\gamma'(t_\infty))
			\end{aligned}
		\end{equation*}
		and the last quantity is as small as we want. So $g_{t_\infty} = g_0$, i.e. $t_\infty \in I$. Moreover $I$ is open: let $t\in I$ and suppose there is a sequence $t_k \to t$ such that $g_{t_k}\neq g_0$. Then
		$$\sfd(\gamma(t), g_{t_k}\gamma'(t)) \leq \sfd(\gamma(t), \gamma(t_k)) + \sfd(g_{t_k}\gamma'(t_k), g_{t_k}\gamma'(t)) \leq 2\vert t - t_k \vert.$$
		When $2\vert t - t_k \vert < \frac{\text{sys}(\Gamma,\gamma'(t))}{2}$ we get $\sfd(g_0^{-1}g_{t_k}\gamma'(t), \gamma'(t)) < \frac{\text{sys}(\Gamma,\gamma'(t))}{2}$, so $g_0^{-1}g_{t_k} = \text{id}$ by Lemma \ref{lemma-inj-radius}. This means that $g_{t_k} = g_0$ for every $t_k$ sufficiently close to $t$. Since $I$ is non empty we conclude that $I=\mathbb{R}$, i.e. $\gamma(t) = g_0 \gamma'(t)$ for every $t\in \mathbb{R}$. So $\gamma = g_0 \gamma'$, i.e. $\bar{\Pi}$ is injective.
		
		The continuity of $\bar{\Pi}^{-1}$ can be checked on sequences since both spaces are metrizable by Lemma \ref{lemma-metric-locgeod} and Proposition  \ref{prop:quotient_of_loc_geod}: observe that the proof of Proposition  \ref{prop:quotient_of_loc_geod} says that also $\Gamma \backslash \LocGeod{\X}$ is metrizable.
		Let $p\colon \LocGeod{\X} \to \Gamma\backslash\LocGeod{\X}$ be the standard projection map induced by the action of $\Gamma$ on $\LocGeod{\X}$. Let $\gamma_k, \gamma_\infty \in \LocGeod{\Gamma \backslash \X}$ be such that $\gamma_k \to \gamma_\infty$ uniformly on compact subsets of $\mathbb{R}$. Let $T\geq 0$ and let $\varepsilon > 0$ be a real number smaller than half the systole of the compact set $\gamma_\infty([-T,T])$. Let $\tilde{\gamma}_k, \tilde{\gamma}_\infty$ be covering local geodesics of $\gamma_k, \gamma_\infty$ respectively. We have $\sfd(\gamma_k(t),\gamma_\infty(t)) < \varepsilon$ for every $t\in [-T,T]$ for $k$ big enough. So, for every $t\in [-T,T]$ there exists a unique $g_k(t)\in \Gamma$ such that $\sfd(g_k(t)\tilde\gamma_k(t),\tilde\gamma_\infty(t)) < \varepsilon$ by Lemma \ref{lemma-inj-radius}.
		Arguing as before we conclude that for every $k$ big enough it holds that $g_k(t) = g_k(0) =:g_k \in \Gamma$ for every $t\in [-T,T]$. This implies that $p(\tilde{\gamma}_k)$ converges uniformly on $[-T,T]$ to $p(\tilde{\gamma}_\infty)$. Since this is true for every  $T\geq 0$ we get the continuity of $\bar\Pi^{-1}$.
	\end{proof}

	\begin{cor}
		\label{cor:quotient_geodesics=local_geodesics_of_quotient}
		Let $\X$ be a proper, geodesic metric space such that every local geodesic of $\X$ is a global geodesic. Let $\Gamma < \Isom(\X)$ be discrete and free. Then $\bar{\Pi}$ is a homeomorphism between $\Gamma \backslash \Geod(\X)$ and $\LocGeod{\Gamma \backslash \X}$. 
		
		The assumptions are satisfied for instance by Busemann convex metric spaces and in particular by $\CAT(-1)$ spaces.
	\end{cor}

	\subsection{The geodesic flow}
	\label{subsec:geodesic-flow}
	
	Let $\X$ be a proper metric space and let $\LocGeod{\X}$ be its space of local geodesic lines.  There is a natural action of $\mathbb{R}$ on $\text{Loc-Geod}(\X)$ defined by reparametrization:
	\begin{equation}
		\label{eq:defin_reparametrization_flow}
		\Phi_t\gamma (\cdot) = \gamma(\cdot + t)
	\end{equation}
	for every $t\in \mathbb{R}$.
	It is a continuous action, i.e. the map $\R \to \text{Homeo}(\LocGeod{\X})$, $t\mapsto \Phi_t$, is a continuous homomorphism of groups. In particular $\Phi_t$ is a homeomorphism of $\LocGeod{\X}$ for every $t\in\mathbb{R}$ and $\Phi_t \circ \Phi_s = \Phi_{t+s}$ for every $t,s\in \mathbb{R}$. The dynamical system $(\LocGeod{\X}, \Phi_1)$ is called the \emph{local geodesic flow of $\X$}. By definition, $\Geod(\X)$ is preserved by the action of $\Phi_t$. The dynamical system $(\Geod(\X), \Phi_1)$ is called the \emph{geodesic flow of $\X$}.
	
	The action of $\Isom(\X)$ on $\Geod(\X)$ commutes with the action $\Phi_t$, in the sense that $g(\Phi_t \gamma) = \Phi_t(g\gamma)$ for every $t\in \R$ and every $\gamma \in \Geod(\X)$. Therefore, $\Phi_t$ induces a continuous action by $\R$ on $\Gamma \backslash \Geod(\X)$ by homeomorphisms, that we denote by $\bar{\Phi}_t$.
	The \emph{quotient geodesic flow of $\Gamma \backslash \X$} is the dynamical system $(\Gamma \backslash \Geod(\X), \bar{\Phi}_1)$.
	This is the flow of which we compute the topological entropy in Theorem \ref{theo:intro-equality-Gromov-hyperbolic}.
	
	For $\CAT(-1)$ spaces, the quotient geodesic flow of $\Gamma \backslash \X$ is conjugated to the local geodesic flow of $\Gamma \backslash \X$.
	
	\begin{cor}
		\label{cor:conjugation_of_flows}
		Let $\X$ be a proper, geodesic metric space such that every local geodesic of $\X$ is a global geodesic. Let $\Gamma < \Isom(\X)$ be discrete and free. Then the map $\bar{\Pi}$ of Proposition \ref{prop-quotient-geodesics} is a conjugation between the dynamical systems $(\Gamma \backslash \Geod(\X), \bar{\Phi}_1)$ and $(\LocGeod{\Gamma \backslash \X}, \Phi_1)$, i.e. $\bar{\Pi}\circ \bar{\Phi}_1 = \Phi_1 \circ \bar{\Pi}$.
	\end{cor}
	\begin{proof}
		Let $p\colon \Geod(\X) \to \Gamma \backslash \Geod(\X)$ be the standard projection map. For every $[\gamma]\in \Gamma \backslash \Geod(\X)$ we have that
		$$(\bar{\Pi}\circ \bar{\Phi}_1)([\gamma(\cdot)]) = \bar{\Pi}([\gamma(1+\cdot)]) = (\pi \circ \gamma)(1+\cdot) = \Phi_1((\pi \circ \gamma)(\cdot)) = (\Phi_1 \circ \bar{\Pi})([\gamma(\cdot)]).$$
	\end{proof}
	
	\section{The entropies of a dynamical system}
	\label{sec:topological_entropy}

	A dynamical system is a couple $(\Y,T)$ where $\Y$ is a topological space and $T\colon \Y \to \Y$ is a continuous map. We will always denote by $\mathscr{B}$ the $\sigma$-algebra of Borel subsets of $\Y$. We denote by $\mathcal{M}_1(\Y,T)$ the set of probability measure on $(\Y,\mathscr{B})$ that are $T$-invariant, i.e. such that $T_\#\mu = \mu$. We recall that a measure $\mu\in \mathcal{M}_1(\Y,T)$ is ergodic if the $\mu$-measure of every subset $A$ of $\Y$ such that $T^{-1}A \subseteq A$ is either $0$ or $1$. We denote the set of ergodic measures by $\mathcal{E}_1(\Y,T)$. The measure-theoretic entropy of the dynamical system $(\Y,T)$ is by definition:
	\begin{equation}
		\label{eq:defin_top_entropy_via_measures}
		h_\text{meas}(\Y,T) := \sup_{\mu\in \mathcal{M}_1(\Y,T)}h_\mu(\Y,T) = \sup_{\mu\in \mathcal{E}_1(\Y,T)}h_\mu(\Y,T),
	\end{equation}
	where $h_\mu(\Y,T)$ denotes the Kolmogorov-Sinai entropy of the measure $\mu$, whose definition we now recall. 
	
	Let $\mu\in \mathcal{M}_1(\Y,T)$.
	A \emph{finite Borel partition} of $\Y$ is a set $\mathcal{Q} = \{Q_1,\ldots,Q_k\}$, where $Q_i \subseteq \Y$ is Borel for every $i = 1,\ldots, k$, $Q_i\cap Q_j = \emptyset$ if $i\neq j$ and $\Y = \bigcup_{i=1}^k Q_i$. The \emph{$\mu$-entropy of $\mathcal{Q}$} is the quantity
	$$H_\mu(\mathcal{Q}) := -\sum_{i=1}^k \mu(Q_i)\log\mu(Q_i).$$
	Here we use the convention that $0\cdot \log 0 = 0$.
	Given $n\in \N$ we define
	$$\bigvee_{j=0}^n T^{-j}\mathcal{Q} := \left\{ Q_{i_0} \cap T^{-1} Q_{i_1} \cap \ldots \cap T^{-n}Q_{i_n} \,:\, i_l \in \{1,\ldots,k\} \right\}.$$
	Notice that $\bigvee_{j=0}^n T^{-j}\mathcal{Q}$ is again a finite Borel partition of $\Y$. So we set
	$$h_\mu(\Y,T,\mathcal{Q}) := \lim_{n\to +\infty} H_\mu\left(\bigvee_{j=0}^n T^{-j}\mathcal{Q}\right).$$
	The limit on the right hand side exists. Finally, the \emph{Kolmogorov-Sinai entropy of $\mu$} is 
	$$h_\mu(\Y,T) := \sup_{Q}h_\mu(\Y,T,\mathcal{Q}),$$
	where the supremum is taken among all possible finite Borel partitions of $\Y$.
	A direct consequence of the definition of the Kolmogorov-Sinai entropy is the following.
	\begin{lemma}
		\label{lemma:measure_entropy_on_measure_1}
		Let $(\Y,T)$ be a dynamical system and let $\mu \in \mathcal{E}_1(\Y,T)$. Suppose that there exists a Borel subset $\Y' \subseteq \Y$ such that $T(\Y') = \Y'$ and $\mu(\Y') = 1$. Then $h_\mu(\Y,T) = h_{\mu \llcorner \Y'}(\Y',T)$, where $\mu \llcorner \Y'$ is the restriction of $\mu$ to $\Y'$.
	\end{lemma}
	\begin{proof}
		Given a finite Borel partition $\mathcal{Q} = \{Q_1,\ldots,Q_k\}$ of $\Y'$ we define $\mathcal{Q}^e := \{Q_1,\ldots,Q_k, \Y\setminus \Y'\}$, which is a finite Borel partition of $\Y$. As collection of sets, $\bigvee_{j=0}^n T^{-j}\mathcal{Q} \subseteq \bigvee_{j=0}^n T^{-j}\mathcal{Q}^e$. From the definitions we have that $h_{\mu \llcorner \Y'}(\Y',T,\mathcal{Q}) \le h_{\mu}(\Y,T,\mathcal{Q}^e)$. Since this is true for every finite Borel partition of $\Y'$ we deduce that $h_{\mu \llcorner \Y'}(\Y',T) \le h_\mu(\Y,T)$.
		
		On the other hand let $\mathcal{Q} = \{Q_1,\ldots,Q_k\}$ be a finite Borel partition of $\Y$ and define $\mathcal{Q}^r := \{Q_1\cap \Y', \ldots, Q_k \cap \Y'\}$, a finite Borel partition of $\Y'$.
		Since $T(\Y\setminus \Y') = \Y\setminus \Y'$ and $\mu(\Y\setminus \Y') = 0$, it follows that $h_\mu(\Y,T,\mathcal{Q}) = h_{\mu\llcorner \Y'}(\Y',T,\mathcal{Q}^r)$. By the arbitrariness of $\mathcal{Q}$ we conclude that $h_\mu(\Y,T) \le  h_{\mu\llcorner \Y'}(\Y',T)$.
	\end{proof}

	Another notion of entropy, introduced by R.Bowen in \cite{Bow73}, is defined for dynamical systems $(\Y,T)$ such that $\Y$ is metrizable. The \emph{topological entropy} of $(\Y,T)$ is defined by
	\begin{equation}
		\label{eq:defin-Bowen-entropy}
		h_\textup{top}(\Y,T) =\inf_{\text{\texthtd}}\sup_{K}\lim_{r\to 0}\lims_{n \to +\infty}\frac{1}{n} \log \textup{Cov}_{\text{\texthtd}^n}(K,r),
	\end{equation}
	where the infimum is taken among all metrics $\text{\texthtd}$ inducing the topology of $\Y$, the supremum is among all compact subsets of $\Y$, $\text{\texthtd}^n$ is the \emph{dynamical metric}
	\begin{equation}
		\label{eq:defin_dynamical_metric}
		\text{\texthtd}^n(y,y') := \max_{i=0,\ldots,n-1} \text{\texthtd}(T^iy,T^iy')
	\end{equation}
	and $\textup{Cov}_{\text{\texthtd}^n}(K,r)$ denotes the minimal number of balls of radius $r$, with respect to the metric $\text{\texthtd}^n$, needed to cover $K$.
	
	The results around the variational principle of \cite{HK95} can be summerized as follows.
	\begin{prop}
		\label{prop:variational_principle}
		Let $(\Y,T)$ be a metrizable dynamical system.
		\begin{itemize}
			\item[(i)] It holds that $h_\textup{meas}(\Y,T) \le h_\textup{top}(\Y,T)$.
			\item[(ii)] If $\Y$ is locally compact then $h_\textup{meas}(\Y,T) = h_\textup{top}(\Y,T)$. 
			\item[(iii)] If $\Y$ is compact, then the right hand side of \eqref{eq:defin-Bowen-entropy} does not depend on the metric \textup{\texthtd}\, inducing the topology of $\Y$. In other words
			$$h_\textup{top}(\Y,T) = \lim_{r\to 0}\lims_{n \to +\infty}\frac{1}{n} \log \textup{Cov}_{\textup{\texthtd}^n}(\Y,r),$$
			for every choice of \textup{\texthtd}.
		\end{itemize}
	\end{prop} 
	
	We observe that in the right hand side of \eqref{eq:defin-Bowen-entropy}, the compact subsets $K$ are not required to be $T$-invariant, where by $T$-invariant we mean $T(K) = K$. For a generic dynamical system, the restriction in \eqref{eq:defin-Bowen-entropy} to $T$-invariant compact subsets gives a strictly smaller quantity, see the remark below \cite[Lemma 1.6]{HK95}. 
	We call it the \emph{invariant topological entropy}, which is by definition
	\begin{equation}
		h_\textup{inv-top}(\Y,T) := \inf_{\text{\texthtd}}\sup_{K\,\,\, T\text{-inv.}}\lim_{r\to 0}\lims_{n \to +\infty}\frac{1}{n} \log \textup{Cov}_{\text{\texthtd}^n}(K,r). 
	\end{equation}
	
	When $K$ is a compact $T$-invariant subset then the dynamical system $(K,T)$ is supported on a compact metrizable space, so the quantity 
	$$\lim_{r\to 0}\lims_{n \to +\infty}\frac{1}{n} \log \textup{Cov}_{\text{\texthtd}^n}(K,r)$$ 
	does not depend on the choice of the metric \texthtd\, and it coincides with the topological entropy of the dynamical system $(K,T)$ by Proposition \ref{prop:variational_principle}.(iii). Therefore
	\begin{equation}
		\label{eq:h_inv-top}
		h_\textup{inv-top}(\Y,T) = \sup_{K\,\,\, T\text{-inv.}} h_\text{top}(K,T).
	\end{equation}

	Let $(\Y,T)$ be a dynamical system and let $\Gamma$ be a group of homeomorphisms of $\Y$ commuting with $T$, i.e. $Tg(y) = gT(y)$ for every $g\in \Gamma$ and every $y\in \Y$. In this situation, the map $T$ induces a continuous function $\bar{T} \colon \Gamma \backslash \Y \to \Gamma \backslash \Y$, $\bar{T}([y]) = [Ty]$ and so we have the dynamical system $(\Gamma \backslash \Y, \bar{T})$.
	\begin{lemma}
		\label{lemma:h_top_finite_quotient}
		Let $(\Y,T)$ be a dynamical system. Suppose $\Y$ is metrizable with a metric $\sfd$. Let $\Gamma < \Isom(\Y,\sfd)$ commuting with $T$. If $\Gamma$ is finite then
		$$h_\textup{inv-top}(\Y,T) = h_\textup{inv-top}(\Gamma\backslash\Y,\bar{T}).$$
	\end{lemma}
	\begin{proof}
		Let $p\colon \Y \to \Gamma \backslash \Y$ be the standard projection. If $K$ is a compact $T$-invariant subset of $\Y$, then $p(K)$ is a compact $\bar{T}$-invariant subset. Vice versa, if $K$ is a compact $\bar{T}$-invariant subset of $\Gamma \backslash \Y$ then $p^{-1}(K)$ is a compact, $T$-invariant subset of $\Y$.
		By Proposition \ref{prop:variational_principle}, we can compute the topological entropy of every compact $T$-invariant subset of $\Y$ using $\sfd$ and the entropy of every compact $\bar{T}$-invariant subset of $\Gamma \backslash \Y$ using the quotient metric $\bar{\sfd}$. 
		Since
		$$\Cov_{\bar{\sfd}^n}(p(K),r) \le \Cov_{\sfd^n}(K,r) \le \#\Gamma \cdot \Cov_{\bar{\sfd}^n}(p(K),r)$$ 
		for every $n\in \N$, every $r>0$ and every $K\subseteq \Y$, we get the thesis.
	\end{proof}

	\subsection{Topological entropy of flows}
	\label{subsec:topological_entropy_flows}
	Let $\Y$ be a metrizable topological space and denote by $\text{Homeo}(\Y)$ the group of self-homeomorphisms of $\Y$ endowed with the compact-open topology. A flow on $\Y$ is a continuous homomorphism $\Phi \colon \R \to \text{Homeo}(\Y)$, $t\mapsto \Phi_t \colon \Y \to \Y$, so $\Phi_0 = \text{id}$ and $\Phi_t \circ \Phi_s = \Phi_{t+s}$. Classically, the topological entropy of a flow is defined via the following formula
	\begin{equation}
		\label{eq:defin_h_top_flow}
		h_\text{top}^\text{flow}(\Y,\Phi) := \inf_{\textup{\texthtd}}\sup_{K} \lim_{r\to 0} \lims_{T \to +\infty} \frac{1}{T} \log \textup{Cov}_{\textup{\texthtd}^T}(K,r),
	\end{equation}
	where the infimum is among all metrics inducing the topology of $\Y$, the supremum is among all compact subsets of $\Y$ and 
	\begin{equation}
		\label{eq:defin_d^T}
		\textup{\texthtd}^T(y,y') = \sup_{t\in [0,T]} \textup{\texthtd}(\Phi_t(y), \Phi_t(y')).
	\end{equation}
	The difference between \eqref{eq:defin_h_top_flow} and \eqref{eq:defin-Bowen-entropy} relies in the different definition of distance in \eqref{eq:defin_d^T} with respect to \eqref{eq:defin_dynamical_metric}. The next lemma focuses on the relation between these two definitions.
	\begin{lemma}
		\label{lemma:h_top_flow_vs_h_top}
		Let $\Y$ be a metrizable topological space and let $\Phi$ be a flow on it. Consider the dynamical system $(\Y,\Phi_1)$. Then $h_\textup{top}(\Y, \Phi_1) \le h_\textup{top}^\textup{flow}(\Y,\Phi)$. Moreover, if $\Y$ is compact and the maps $\{\Phi_t\}_{t\in [0,1]}$ are uniformly equicontinuous then $h_\textup{top}(\Y, \Phi_1) = h_\textup{top}^\textup{flow}(\Y,\Phi)$.
	\end{lemma}
	The uniform equicontinuity means that for every $\varepsilon > 0$ there exists $\delta > 0$ such that if $\text{\texthtd}(y,y') < \delta$ then $\text{\texthtd}(\Phi_t(y), \Phi_t(y')) < \varepsilon$ for every $y,y'\in \Y$ and every $t\in [0,1]$. Here $\text{\texthtd}$ is any metric inducing the topology of $\Y$. By compactness of $\Y$ and of $[0,1]$ the notion of uniform continuity does not depend on the specific metric $\text{\texthtd}$.
	\begin{proof}
		The inequality follows from \eqref{eq:defin-Bowen-entropy} and \eqref{eq:defin_h_top_flow} since $\text{\texthtd}^T(y,y') \ge \text{\texthtd}^{\lfloor T \rfloor}(y,y')$ for every $y,y'\in \Y$. On the other hand, if $\Y$ is compact we have that
		$$h_\text{top}(\Y,\Phi_1) = \lim_{r\to 0} \lims_{n \to +\infty} \frac{1}{n} \log \text{Cov}_{\text{\texthtd}^n}(\Y,r)$$
		where $\text{\texthtd}$ is any metric inducing the topology of $\Y$, by Proposition \ref{prop:variational_principle}, while
		$$h_\text{top}^\text{flow}(\Y,\Phi) \le \lim_{r\to 0} \lims_{T \to +\infty} \frac{1}{T} \log \text{Cov}_{\text{\texthtd}^T}(\Y,r).$$
		Let $\varepsilon > 0$ and let $\delta > 0$ be as in the definition of uniform equicontinuity of the maps $\{\Phi_t\}_{t\in [0,1]}$.
		Let $\{y_i\}$ be a subset of $\Y$ realizing $\text{Cov}_{\text{\texthtd}^{\lfloor T \rfloor}}(\Y,\delta)$. We claim that this set is $\varepsilon$-dense with respect to $\text{\texthtd}^T$. Fix $y\in \Y$ and take $i$ such that $\text{\texthtd}^{\lfloor T \rfloor}(y,y_i) < \delta$. For every $t\in [0,T]$ we consider $n_t := \lfloor t \rfloor$ and we observe that $$\text{\texthtd}(\Phi_t(y),\Phi_t(y_i)) =  \text{\texthtd}(\Phi_{t-n_t}(\Phi_{n_t}(y)),\Phi_{t-n_t}(\Phi_{n_t}(y_i))).$$
		Since $\text{\texthtd}(\Phi_{n_t}(y), \Phi_{n_t}(y_i)) < \delta$ and since $t-n_t \in [0,1]$ we deduce that $\text{\texthtd}(\Phi_t(y),\Phi_t(y_i)) < \varepsilon$. Since this is true for every $t\in [0,T]$ we have that $\text{\texthtd}^T(y,y_i) < \varepsilon$. Then $\Cov_{\text{\texthtd}^T}(\Y,\varepsilon) \le \Cov_{\text{\texthtd}^{\lfloor T \rfloor}}(\Y,\delta)$. The thesis follows.
	\end{proof}
	
	\subsection{Topological entropy of the geodesic flow}
	\label{subsec:topological_entropy_geodesic_flow}
	
	We specialize the discussion to the case of the geodesic flow. Let $\X$ be a proper metric space. Remember that the geodesic flow of $\X$ is the dynamical system $(\Geod(\X), \Phi_1)$, where $\Phi_1$ is defined in \eqref{eq:defin_reparametrization_flow}. Let $\Gamma < \Isom(\X)$ be discrete. The quotient geodesic flow of $\Gamma \backslash \X$ is the dynamical system $(\Gamma \backslash \Geod(\X), \bar{\Phi}_1)$. This is the flow whose topological entropy appears in Theorem \ref{theo:intro-equality-Gromov-hyperbolic}. By Proposition \ref{prop:quotient_of_loc_geod}, $\Gamma \backslash \Geod(\X)$ is a locally compact metrizable space, so $h_\text{meas}(\Gamma \backslash \Geod(\X), \bar{\Phi}_1) = h_\text{top}(\Gamma \backslash \Geod(\X), \bar{\Phi}_1)$, by Proposition \ref{prop:variational_principle}.
	
	
	The famous Birkhoff's ergodic theorem, in the case of the quotient geodesic flow, reads as follows.
	\begin{prop}
		\label{prop-Birkhoff}
		Let $\X$ be a proper metric space and let $\Gamma < \textup{Isom}(\X)$ be discrete. Let $\mu \in \mathcal{E}_1(\Gamma \backslash \Geod(\X), \bar{\Phi}_1)$. Then, for every $f \in L^1(\mu)$ it holds that
		\begin{equation}
			\label{Birkhoff}
			\lim_{N\to +\infty}\frac{1}{N}\sum_{j=0}^{N-1} (f \circ \bar{\Phi}_j)([\gamma]) = \int f\,d\mu
		\end{equation}
		for $\mu$-a.e. $[\gamma] \in \Gamma \backslash \X$. In other words, the limit in \eqref{Birkhoff} exists for $\mu$-a.e. $[\gamma] \in \Gamma \backslash \X$ and equals the right hand side.
	\end{prop}
	
	Since conjugated dynamical systems have same topological entropy, then 
	$$h_\text{top}(\Gamma \backslash \Geod(\X), \bar{\Phi}_1) = h_\text{top}(\LocGeod{\Gamma\backslash \X}, \Phi_1)$$
	under the assumptions of Corollary \ref{cor:conjugation_of_flows}. In particular this applies to proper $\CAT(-1)$ spaces $\X$ and torsion-free, discrete groups $\Gamma < \Isom(\X)$.
	
	\section{$f$-entropy on Gromov-hyperbolic metric spaces}
	\label{sec:Lip-Bowen-entropy}
	
	In order to prove Theorem \ref{theo:intro-equality-Gromov-hyperbolic} we need a way to connect the topological entropy of the dynamical system $(\Gamma \backslash \Geod(\X), \bar{\Phi}_1)$ to the topological entropy of the dynamical system $(\Geod(\X), \Phi_1)$. However, if one uses the definition of topological entropy of $(\Geod(\X), \Phi_1)$ given in Section \ref{sec:topological_entropy}, then $h_\text{top}(\Geod(\X), \Phi_1) = 0$, see \cite[Lemma 4.1]{Cav21}. In \cite{Cav21} we proposed a modification of the definition of the topological entropy of $(\Geod(\X), \Phi_1)$ introducing the Lipschitz-topological entropy. In this paper we will use a slightly different notion that simplifies the exposition.
	
	Let $\X$ be a proper, geodesic, Gromov-hyperbolic metric space and let $x\in \X$. For every $A \subseteq \X$ and every $C\subseteq \partial \X$ we set
	\begin{equation}
		\Geod(A;C) := \left\{ \gamma \in \Geod(\X)\,:\, \gamma(0) \in A,\, \gamma^\pm \in C\right\}.
	\end{equation} 
	This notation is coherent with the one introduced in \eqref{eq:MD<e_T_Geod}.
	Let $f\in \mathcal{F}$ as introduced in Section \ref{subsec:space_local_geodesics}. We define the \emph{upper $f$-entropy of $C$} as
	\begin{equation}
		\label{eq:defin-upper-Lip-Bowen-entropy}
		\overline{h}_{f}(C) := \sup_{R\ge 0} \lim_{r\to 0} \lims_{T \to +\infty} \frac{1}{T} \log \text{Cov}_{\sfD_f^T}(\Geod(\overline{B}(x,R);C),r)
	\end{equation}
	and the \emph{lower $f$-entropy of $C$} as
	\begin{equation}
		\label{eq:defin-lower-Lip-Bowen-entropy}
		\underline{h}_{f}(C) :=  \sup_{R \ge 0} \lim_{r\to 0} \limi_{T \to +\infty} \frac{1}{T} \log \text{Cov}_{\sfD_f^T}(\Geod(\overline{B}(x,R);C),r),
	\end{equation}
	where the dynamical distance $\sfD_f^T$ is
	$$\sfD_f^T(\gamma, \gamma') = \sup_{t\in [0,T]} \sfD_f(\Phi_t \gamma, \Phi_t \gamma').$$
	This last definition is the same as the one in \eqref{eq:defin_d^T}. The upper and lower $f$-entropies do not depend on the choice of the basepoint $x\in \X$.
	
	\begin{prop}
		\label{prop:h_Lip_top_>=_MD}
		Let $\X$ be a proper, geodesic, Gromov-hyperbolic metric space, let $C\subseteq \partial \X$ and let $f\in \mathcal{F}$. Then
		$$\overline{h}_{f}(C) \ge  \overline{\MD}(C)\quad \text{and}\quad \underline{h}_{f}(C) \ge  \underline{\MD}(C).$$
	\end{prop}
	\begin{proof}
		We give the proof for the upper entropy and the upper Minkowski dimension only, the other case being analogous. Let $\delta$ be the hyperbolicity constant of $\X$.
		Let $x\in \X$ and let $M:= \sfd(x,\text{QC-Hull}(C)) + 22\delta$. Notice that 
		$$\sfd(e_T(\gamma),e_T(\gamma'))=\sfd(\gamma(T), \gamma'(T)) \le \sfD_f(\Phi_T\gamma, \Phi_T \gamma') \le \sfD_f^T(\gamma,\gamma')$$
		for every $\gamma,\gamma'\in \Geod(\X)$ by Lemma \ref{lemma-metric-locgeod}. This implies that
		\begin{equation}
			\begin{aligned}
				\lims_{T \to +\infty}\frac{1}{T}\log \Cov_{\sfD_f^T}(\Geod(\overline{B}(x,R);C), r) \ge \lims_{T \to +\infty}\frac{1}{T}\log \Cov_{\sfd}(e_T(\Geod(\overline{B}(x,M);C)), r)
			\end{aligned}
		\end{equation}
		for every $r>0$ and every $R \ge M$. The thesis follows by \eqref{eq:MD<e_T_Geod}.
	\end{proof}
	
	The opposite inequalities are false in general. They are true under additional assumptions, as we will see in Section \ref{sec:proof_of_h_top_<h_Gamma}.
	
	\section{The proof of Theorem \ref{theo:intro-h_top>=h_Gamma}}
	\label{sec:proof_of_main_theorems}
	
	In this section we prove Theorem \ref{theo:intro-h_top>=h_Gamma}. Let us recall the setup. Let $\X$ be a proper, geodesic, Gromov-hyperbolic metric space, let $\Gamma < \Isom(\X)$ be discrete, non-elementary and virtually torsion-free and let $x\in \X$ be a basepoint.
	The goal is to show that $h_\text{top}(\Gamma \backslash \Geod(\X), \bar{\Phi}_1) \ge h_\text{crit}(\X;\Gamma)$. We first reduce the proof to the torsion-free case.
	\begin{lemma}
	\label{lemma:reduction_torsion_free}
		Under the assumptions of Theorem \ref{theo:intro-h_top>=h_Gamma}, let $\Gamma_0 \triangleleft \Gamma$ be a finite index, torsion-free subgroup. Then $h_\textup{crit}(\X;\Gamma_0) = h_\textup{crit}(\X;\Gamma)$ and $h_\textup{inv-top}(\Gamma_0 \backslash \Geod(\X), \bar{\Phi}_1) = h_\textup{inv-top}(\Gamma \backslash \Geod(\X), \bar{\Phi}_1)$.
	\end{lemma}
	\begin{proof}
		The statement about the critical exponent is known. As described in Section \ref{sec:topological_entropy}, the dynamical system $(\Gamma \backslash \Geod(\X), \bar{\Phi}_1)$ coincides with the quotient of the dynamical system $(\Gamma_0 \backslash \Geod(\X), \bar{\Phi}_1)$ modulo the finite group $\Gamma / \Gamma_0$. In particular the action of $\Gamma/\Gamma_0$ commutes with $\bar{\Phi}_1$. Moreover, this dynamical system is metrizable via the metric $\bar{\sfD}_f$ introduced in Section \ref{subsec:space_local_geodesics}, with respect to which $\Gamma / \Gamma_0$ acts by isometries, see Proposition \ref{prop:quotient_of_loc_geod}. The thesis follows by Lemma \ref{lemma:h_top_finite_quotient}.
	\end{proof}
	
	The strategy to prove Theorem \ref{theo:intro-h_top>=h_Gamma} relies on the following chain of inequalities:
	\vspace{2mm}
	\begin{equation}
		\begin{aligned}
			h_\text{top}(\Gamma \backslash \Geod(\X), \bar{\Phi}_1) \ge h_\text{inv-top}(\Gamma \backslash \Geod(\X), \bar{\Phi}_1) \ge \sup_{\tau \ge 0} \underline{h}_f(\Lambda_\tau) \ge \sup_{\tau \ge 0} \underline{\MD}(\Lambda_\tau) \ge h_\text{crit}(\X;\Gamma),
		\end{aligned}
	\end{equation}
	where the sets $\Lambda_\tau \subseteq \partial \X$ have been defined in Section \ref{subsec:BJ}. The first inequality is true by definition, as we recalled in Section \ref{sec:topological_entropy}.
	The third inequality is Proposition \ref{prop:h_Lip_top_>=_MD}. The last inequality is Theorem \ref{theo:Bishop-Jones}. We have reduced the proof of Theorem \ref{theo:intro-h_top>=h_Gamma} to the proof of
	\begin{equation}
		\label{eq:for_thm_h_top_>_h_Gamma}
		h_\text{inv-top}(\Gamma \backslash \Geod(\X), \bar{\Phi}_1) \ge \sup_{\tau \ge 0} \underline{h}_f(\Lambda_\tau),
	\end{equation}
	and we can further assume that $\Gamma$ is torsion-free by Lemma \ref{lemma:reduction_torsion_free}.
	
	We need to identify the compact, $\bar{\Phi}_1$-invariant subsets of $\Gamma \backslash \Geod(\X)$. For every $\tau \ge 0$ we define
	$$K_\tau = \lbrace [\gamma] \in \Gamma \backslash \Geod(\X) \,:\, \sfd(\gamma(n), \Gamma x) \le \tau \text{ for every } n\in \mathbb{Z} \rbrace.$$
	We observe that the set is well posed, since the condition on the representative of $[\gamma]$ is $\Gamma$-invariant.
	Under the assumptions of Corollary \ref{cor:conjugation_of_flows}, the sets $K_\tau$ are comparable via $\bar{\Pi}$ to the set of local geodesic lines of $\Gamma \backslash \X$ that are contained in a compact region of $\Gamma \backslash \X$.
	\begin{lemma}
		\label{lemma:invariant_compact_subsets}
		$K_\tau$ is a compact, $\bar{\Phi}_1$-invariant subset for every $\tau \geq 0$. Moreover every compact, $\bar{\Phi}_1$-invariant subset of $\Gamma \backslash \Geod(\X)$ is contained in $K_\tau$ for some $\tau \geq 0$.
	\end{lemma}
	\begin{proof}
		For every $\tau \geq 0$, the set $K_\tau$ satisfies $\bar{\Phi}_1(K_\tau)\subseteq K_\tau$ and $\bar{\Phi}_1^{-1}(K_\tau)\subseteq K_\tau$, so it is $\bar{\Phi}_1$-invariant. Let $\{[\gamma_k]\}_k\subseteq K_\tau$ be a sequence. By definition, for every $k$ there exists $g_k\in \Gamma$ such that $\sfd(g_k\gamma_k(0), x) \le \tau$. By Ascoli-Arzelà's Theorem, the sequence $g_k\gamma_k$ subconverges to a geodesic $\gamma_\infty \in \Geod(\X)$. Therefore $[\gamma_k]$ subconverges to $[\gamma_\infty]$. Moreover, for every $n\in \Z$ we have that $\sfd(g_k\gamma_k(n),\Gamma x) \le \tau$, implying that $\sfd(\gamma_\infty(n), \Gamma x) \le \tau$, i.e. $[\gamma_\infty] \in K_\tau$. This completes the proof of the first statement.
		
		Let $K$ be a compact, $\bar{\Phi}_1$-invariant subset of $\Gamma \backslash \Geod(\X)$.
		Suppose that $K$ is not contained in $K_\tau$ for every $\tau \ge 0$. Then, for every $k \geq 0$ there exists $[\gamma_k] \in K$ such that $\sfd(\gamma_k(n_k), \Gamma x) > k$ for some $n_k\in \mathbb{Z}$.  For every $k$ we reparametrize $\gamma_k$ so that $\sfd(\gamma_k(0), \Gamma x) > k$. Since $K$ is $\bar{\Phi}_1$-invariant then the equivalence classes of the reparametrized geodesics belong again to $K$. Since $K$ is compact we can find a subsequence, denoted again by $[\gamma_k]$, that converges to $[\gamma_\infty] \in K$. This means that there exist $g_k \in \Gamma$ such that $g_k\gamma_k$ converges to $\gamma_\infty$, in particular $g_k\gamma_k(0)$ converges to $\gamma_\infty(0)$. But $\sfd(g_k\gamma_k(0), \Gamma x) > k$ for every $k$, which is a contradiction.
	\end{proof}

	\begin{proof}[Proof of Theorem \ref{theo:intro-h_top>=h_Gamma}]
		As discussed above, the proof is reduced to show \eqref{eq:for_thm_h_top_>_h_Gamma} under the additional assumption that $\Gamma$ is torsion-free. By Lemma \ref{lemma:invariant_compact_subsets} we have that
		\begin{equation}
			\label{eq:h_invtop=sup_h_top}
			h_\text{inv-top}(\Gamma \backslash \Geod(\X), \bar{\Phi}_1) = \sup_{\tau \geq 0} h_{\text{top}}(K_\tau, \bar{\Phi}_1).
		\end{equation}
		The proof of \eqref{eq:for_thm_h_top_>_h_Gamma} will be a consequence of the following fact. For every $\tau \ge 0$, for every $f\in \mathcal{F}$, and for every $R\ge 0$ it holds that
		\begin{equation}
			\begin{aligned}
			\label{eq:inequality_Lambda_tau_up_down}
			\lim_{r\to 0}\limi_{T \to +\infty} \frac{1}{T}\log \Cov_{\sfD_f^T}(\Geod(\overline{B}(x,R);\Lambda_\tau),r) &\le h_\text{top}^{\text{flow}}(K_{2\tau + 8\delta + R}, \bar{\Phi})\\
			&=  h_\text{top}(K_{2\tau + 8\delta + R}, \bar{\Phi}_1),
			\end{aligned}
		\end{equation}
		where the last equality follows by Lemma \ref{lemma:h_top_flow_vs_h_top} and the fact that the maps $\{\bar{\Phi}_t\}_{t\in [0,1]}$ are uniformly continuous with respect to the metric $\bar{\sfD}_{e^{-\vert s \vert}}$. Indeed, for $t\in [0,1]$ we have that 
		\begin{equation}
			\begin{aligned}
				\sfD_{e^{-\vert s \vert}}(\Phi_t(\gamma), \Phi_t(\gamma')) = \sup_{s \in \R} \sfd(\gamma(s+t), \gamma'(s+t))e^{-\vert s \vert} = \sfD_{e^{-\vert s \vert}}(\gamma, \gamma')\sup_{s \in \R} \frac{e^{-\vert s \vert}}{e^{-\vert s + t \vert}} = e^2 \sfD_{e^{-\vert s \vert}}(\gamma, \gamma').
			\end{aligned}
		\end{equation}
		This means that the maps $\{\Phi_t\}_{t\in [0,1]}$ are uniformly Lipschitz with respect to $\sfD_{e^{-\vert s \vert}}$, so the quotient maps $\{\bar{\Phi}_t\}_{t\in [0,1]}$ are uniformly continuous with respect to the metric $\bar{\sfD}_{e^{-\vert s \vert}}$.
		
		We observe that if \eqref{eq:inequality_Lambda_tau_up_down} holds, then we can take the supremum over $R \ge 0$ to get that
		$$\underline{h}_f(\Lambda_\tau) \le h_\text{inv-top}(\Gamma \backslash \Geod(\X), \bar{\Phi}_1)$$
		for every $\tau \ge 0$. Taking the supremum over $\tau \ge 0$ we finally obtain \eqref{eq:for_thm_h_top_>_h_Gamma}. 
		
		It remains to prove the inequality in \eqref{eq:inequality_Lambda_tau_up_down}. We fix $\tau, R \ge 0$ and $f\in \mathcal{F}$. By Proposition \ref{prop:quotient_of_loc_geod} the metric $\sfD_f$ defines a quotient metric $\bar{\sfD}_f$ on $\Gamma \backslash \Geod(\X)$ and that can be used for the computation of the topological entropy of the compact dynamical system $(K_{2\tau + 8\delta + R}, \bar{\Phi}_1)$, see Proposition \ref{prop:variational_principle}. We fix $T\ge 0$. Let $\{[\gamma_i]\}$ be a subset of $K_{2\tau + 8\delta + R}$ realizing $\Cov_{\bar{\sfD}_f^T}(K_{2\tau + 8\delta + R}, r)$. We define
		$$A_T:=\left\{\gamma \in \Geod(\X) \, : \, [\gamma] = [\gamma_i] \text{ for some } i \text{ and } \sfd(\gamma(0),x) \le R+r\right\}.$$
		
		The group $\Gamma$ acts by isometries on $(\Geod(\X), \sfD_f)$ by Proposition \ref{prop:quotient_of_loc_geod}. Let $\sigma > 0$ be the systole of $\Gamma$ on the compact set $K_{2\tau + 8\delta + R}$: it is positive by Lemma \ref{lemma:systole_positive}. For every two $\gamma,\gamma' \in \Geod(\X)$ such that $[\gamma]=[\gamma']=[\gamma_i]$ we have that either $\gamma=\gamma'$ or $\sfD_f(\gamma,\gamma') > \frac{\sigma}{2}$, by Lemma \ref{lemma-inj-radius}. This implies that either $\gamma = \gamma'$ or $\sfd(\gamma(0), \gamma'(0)) > \frac{\sigma}{2}$ by Lemma \ref{lemma-metric-locgeod}.
		Therefore
		$$\# A_T \le \Cov_{\bar{\sfD}_f^T}(K_{2\tau + 8\delta + R}, r) \cdot \Pack_\sfd\left(\overline{B}(x,R+r), \frac{\sigma}{2}\right).$$
		Observe that the second factor does not depend on $T$. So, inequality \eqref{eq:inequality_Lambda_tau_up_down} is true if we show that $A_T$ is $r$-dense in $\Geod(\overline{B}(x,R); \Lambda_\tau)$ with respect to $\sfD_f^T$ for every $r< \frac{\sigma}{2}$. Indeed, in that case we have that
		$$\lim_{r\to 0}\limi_{T \to +\infty} \frac{1}{T}\log \Cov_{\sfD_f^T}(\Geod(\overline{B}(x,R);\Lambda_\tau),r) \le \lim_{r\to 0}\limi_{T \to +\infty} \frac{1}{T}\log  \Cov_{\bar{\sfD}_f^T}(K_{2\tau + 8\delta + R}, r).$$
		The right hand side is exactly $h_\text{top}^\text{flow}(K_{2\tau + 8\delta + R}, \bar{\Phi})$.
		
		It remains to prove that $A_T$ is $r$-dense with respect to $\sfD_f^T$ in $\Geod(\overline{B}(x,R); \Lambda_\tau)$ for $r<\frac{\sigma}{2}$. Let $\gamma \in \Geod(\overline{B}(x,R); \Lambda_\tau)$. By definition, there exist geodesic rays $\xi_{x,\gamma^\pm}$ such that $\sfd(\xi_{x,\gamma^\pm}(n), \Gamma x) \le 2\tau$ for every $n\in \N$. By Lemma \ref{parallel-geodesics} we deduce that $\sfd(\gamma(n), \Gamma x) \le 2\tau + 8\delta + R$ for every $n\in \Z$, because $\sfd(\gamma(0),x) \le R$. By definition, $[\gamma] \in K_{2\tau + 8\delta + R}$. Hence we can find $i$ such that $\bar{\sfD}_f^T([\gamma],[\gamma_i]) \le r$. 
		By definition of quotient metric, for every $t\in [0,T]$ there exists $g_t \in \Gamma$ such that $$\sfD_f(\Phi_t(\gamma), g_t \Phi_t(\gamma_i)) = \sfD_f(\Phi_t(\gamma),  \Phi_t (g_t\gamma_i)) \le r.$$
		The isometry $g_t \in \Gamma$ is unique by Lemma \ref{lemma-inj-radius}, 	since $r < \frac{\sigma}{2}$. Notice that $\sfd(\gamma(0), g_0\gamma_i(0)) \le r$ because of Lemma \ref{lemma-metric-locgeod}, so $\sfd(g_0\gamma_i(0), x) \le R+r$. In other words, $g_0\gamma_i \in A_T$. We consider the set $I=\{t \in [0,T] \,: \, g_t = g_0\}$ and we show that it is open and closed. Suppose $t_k \to t_\infty$ with $t_k \in I$. Then $\sfD_f(\Phi_{t_k}(\gamma), \Phi_{t_k}(g_{0}\gamma_i)) \le r$ for every $k$. Therefore, $\sfD_f(\Phi_{t_\infty}(\gamma), \Phi_{t_\infty}(g_{0}\gamma_i)) \le r$, showing that $I$ is closed. On the other hand let $t\in I$ and suppose that there exists a sequence $t_k \to t$ with $g_{t_k} \neq g_0$. We have that $\sfd(g_{t_k}\gamma_i(0), \gamma_i(0)) \le 2t_k + 2r$, so by discreteness we can find a constant subsequence, that we denote again by $g_{t_k}$, such that $g_{t_k} = g' \neq g_0$. But this implies that $\sfD_f(\Phi_t (\gamma), \Phi_t (g' \gamma_i)) \le r$, so $g' = g_0$ by uniqueness. This contradiction shows that $I$ is open, thus $I=[0,T]$. Therefore $g_0\gamma_i \in A_T$ is such that $\sup_{t\in [0,T]}\sfD_f(\Phi_t (\gamma), \Phi_t (g_0\gamma_i)) \le r$. This concludes the proof.
	\end{proof}

	\section{The proof of Theorem \ref{theo:intro-equality-Gromov-hyperbolic}}
	\label{sec:proof_of_h_top_<h_Gamma}
	In this section we prove Theorem \ref{theo:intro-equality-Gromov-hyperbolic}. Let $\X$ be a proper, line-convex, Gromov-hyperbolic metric space and let $\Gamma < \Isom(\X)$ be discrete and non-elementary. We suppose that $h_\text{erg}(\X; \Gamma) < +\infty$. Our goal is to show that $h_\text{top}(\Gamma \backslash \Geod(\X), \bar{\Phi}_1) \le h_\text{crit}(\X;\Gamma)$, which is sufficient in view of Theorem \ref{theo:intro-h_top>=h_Gamma}.
	
	Recall that the space $\Gamma \backslash \Geod(\X)$ is locally compact, by Proposition \ref{prop:quotient_of_loc_geod}. So, by Proposition \ref{prop:variational_principle}, the topological entropy of $(\Gamma \backslash \Geod(\X), \bar{\Phi}_1)$ satisfies
	$$h_\text{top}(\Gamma \backslash \Geod(\X), \bar{\Phi}_1) = \sup_{\mu\in \mathcal{E}_1(\Gamma \backslash \Geod(\X), \bar{\Phi}_1)} h_\mu(\Gamma \backslash \Geod(\X), \bar{\Phi}_1).$$
	The strategy is the following. Recall the subsets $\Lambda_{\tau,c,\varepsilon,n}$ defined in Section \ref{subsec:BJ}. We fix a measure $\mu\in \mathcal{E}_1(\Gamma \backslash \Geod(\X), \bar{\Phi}_1)$. Then we find $c > 0$ with the following property: for every $\varepsilon > 0$ there exist $\tau_\varepsilon > 0$ and $n_\varepsilon\in \N$ such that $\mu(\Gamma \backslash \Geod(\X;\Lambda_{\tau,c, \varepsilon, n_\varepsilon} \cap \Lambda_\text{erg})) = 1$. Next, we prove following chain of inequalities:
	\begin{equation}
		\label{eq:strategy_for_h_top_<h_Gamma}
		\begin{aligned}
			h_\mu(\Gamma \backslash \Geod(\X), \bar{\Phi}_1) &= h_\mu(\Gamma \backslash \Geod(\X;\Lambda_{\tau_\varepsilon,c,\varepsilon, n_\varepsilon} \cap \Lambda_{\text{erg}}), \bar{\Phi}_1) \\
			&\le h_\text{top}(\Gamma \backslash \Geod(\X;\Lambda_{\tau_\varepsilon,c,\varepsilon,n_\varepsilon} \cap \Lambda_{\text{erg}}), \bar{\Phi}_1)\\
			& \le \overline{h}_f(\Lambda_{\tau_\varepsilon,c,\varepsilon, n_\varepsilon} \cap \Lambda_{\text{erg}}) \\
			&\le \overline{\MD}(\Lambda_{\tau_\varepsilon,c,\varepsilon,n_\varepsilon} \cap \Lambda_{\text{erg}}) + h_\text{erg}(\X; \Gamma)\cdot\frac{2\varepsilon}{c-\varepsilon} \\
			&\le h_\text{crit}(\X;\Gamma)\cdot \frac{c+\varepsilon}{c-\varepsilon} + h_\text{erg}(\X;\Gamma)\cdot\frac{2\varepsilon}{c-\varepsilon}.
		\end{aligned}
	\end{equation}
	By taking $\varepsilon \to 0$ we conclude that $	h_\mu(\Gamma \backslash \Geod(\X), \bar{\Phi}_1) \le h_\text{crit}(\X;\Gamma)$ for every ergodic measure $\mu \in \mathcal{E}_1(\Gamma \backslash \Geod(\X), \bar{\Phi}_1)$, yelding the proof.
	
	Let us start with the first step, which is a more accurate version of \cite[Theorem C]{Cav24} and it is based on Birkhoff's ergodic theorem.

	\begin{prop}
		\label{prop:Birkhoff_mean}
		Let $\X$ be a proper, geodesic, Gromov-hyperbolic metric space, let $x\in \X$ and let $\Gamma < \Isom(\X)$ be discrete. For every $\mu \in \mathcal{E}_1(\Gamma \backslash \Geod(\X), \bar{\Phi}_1)$ there exists $c\in [1,+\infty)$ satisfying the following. For every $\varepsilon > 0$ there exist $\tau_\varepsilon > 0$ and $n_\varepsilon \in \N$ such that $\mu(\Gamma \backslash \Geod(\X;\Lambda_{\tau_\varepsilon,c, \varepsilon, n_\varepsilon} \cap \Lambda_\textup{erg})) = 1$.
	\end{prop}
	\begin{proof}
		Let $\{x_i\}_{i\in\mathbb{N}}\subseteq \X$ be a countable set such that $\X = \bigcup_{i\in\mathbb{N}} B(x_i,1)$. For every $i$ let
		$$V_i := \{ \gamma \in \Geod(\X)\,:\, \gamma(0) \in B(x_i,1) \}$$
		and
		$$U_i := \Gamma \backslash V_i \subseteq \Gamma \backslash \Geod(\X).$$
		Since $\{V_i\}_{i\in \mathbb{N}}$ is a covering of $\Geod(\X)$, also $\{U_i\}_{i\in\mathbb{N}}$ is a covering of $\Gamma \backslash \Geod(\X)$. Then, there exists $i_0\in \mathbb{N}$ such that $\mu(U_{i_0}) = \sigma > 0$.
		To every $[\gamma] \in \Gamma \backslash \Geod(\X)$ we associate the set of integers $\Theta([\gamma]) = \lbrace \vartheta_i([\gamma])\rbrace$ defined recursively by
		$$\vartheta_0([\gamma])=0, \qquad \vartheta_{i+1}([\gamma]) = \min \lbrace n\in\mathbb{N}, n > \vartheta_i([\gamma]) \text{ s.t. } \bar{\Phi}_n([\gamma]) \in U_{i_0}\rbrace.$$
		We apply Proposition \ref{prop-Birkhoff} to the indicator function of the set $U_{i_0}$, namely $\chi_{U_{i_0}}$, obtaining that for $\mu$-a.e.$[\gamma] \in \Gamma \backslash \Geod(\X)$ it holds that
		$$\exists \lim_{N\to +\infty}\frac{1}{N}\sum_{j=0}^{N-1} (\chi_{U_{i_0}} \circ \bar{\Phi}_j)([\gamma]) = \mu(U_{i_0}) = \sigma \in (0,1].$$
		Observe that $(\chi_{U_{i_0}} \circ \bar{\Phi}_j)([\gamma]) = 1$ if and only if $j\in \Theta([\gamma])$ and it is $0$ otherwise. So 
		$$\lim_{N\to +\infty}\frac{1}{N}\sum_{j=0}^{N-1} (\chi_{U_{i_0}} \circ \bar{\Phi}_j)([\gamma]) = \lim_{N\to +\infty}\frac{\#\Theta([\gamma]) \cap [0,N-1]}{N},$$
		and the right hand side is by definition the density of the set $\Theta([\gamma])$. Given the standard increasing enumeration $\lbrace \vartheta_0([\gamma]), \vartheta_1([\gamma]),\ldots \rbrace$ of $\Theta([\gamma])$, it holds that
		$$\lim_{N\to +\infty}\frac{\#\Theta([\gamma]) \cap [0,N-1]}{N} = \lim_{N\to +\infty}\frac{N}{\vartheta_N([\gamma])}.$$
		So, for $\mu$-a.e.$[\gamma] \in \Gamma \backslash \Geod(\X)$ we have that
		\begin{equation}
			\label{eq-rec-times}
			\exists \lim_{N\to +\infty}\frac{\vartheta_N([\gamma])}{N} = \frac{1}{\sigma} \in [1,+\infty).
		\end{equation}
		In the same way, applying the same argument to the flow at time $-1$, we get that for $\mu$-a.e.$[\gamma] \in \Gamma \backslash \Geod(\X)$ we have
		\begin{equation}
			\label{eq-rec-times-negative}
			\exists \lim_{N\to +\infty}\frac{\vartheta_N([-\gamma])}{N} = \frac{1}{\sigma} \in [1,+\infty).
		\end{equation}
		Here $-\gamma$ denotes the curve $-\gamma(t) := \gamma(-t)$. We deduce that \eqref{eq-rec-times} and \eqref{eq-rec-times-negative} hold together for $\mu$-a.e.$[\gamma] \in \Gamma \backslash \Geod(\X)$. 
		
		We claim that the proposition holds with $c := \frac{1}{\sigma} \in [1,+\infty)$. We fix $\varepsilon > 0$. Let $[\gamma] \in \Gamma \backslash\Geod(\X)$ be such that \eqref{eq-rec-times} and \eqref{eq-rec-times-negative} holds.
		We notice that an integer $n$ satisfies $n \in \Theta([\gamma])$ if and only if there exists a representative $g\gamma$ of $[\gamma]$, with $g\in \Gamma$, such that $\Phi_n(g\gamma) \in V_{i_0}$, i.e. $g\gamma(n) \in B(x_{i_0},1)$. In other words $n \in \Theta([\gamma])$ if and only if
		\begin{equation}
			\label{eq:returning_ergodic}
			\sfd(\gamma(n), \Gamma x_{i_0}) < 1.
		\end{equation}
		We fix a geodesic ray $\xi_{x, \gamma^+}$. By Lemma \ref{parallel-geodesics} we have that $d(\xi_{x,\gamma^+}(t), \gamma(t)) \leq 8\delta + \sfd(x,\gamma(0))$ for every $t\geq 0$. This, together with \eqref{eq:returning_ergodic} says that $d(\xi_{x,\gamma^+}(\vartheta_N([\gamma])), \Gamma x_{i_0}) < 8\delta + \sfd(x,\gamma(0)) + 1$. By definition this means that $\gamma^+ \in \Lambda_{\tau, \Theta([\gamma])}$, where $\tau = 8\delta + \sfd(x,\gamma(0)) + 1$.
		By the properties of the sequence $\Theta([\gamma])$ proved in \eqref{eq-rec-times}, we have that
		$$\gamma^+ \in \Lambda_\text{erg} \cap \bigcup_{\tau > 0} \bigcup_{n\in \N} \Lambda_{\tau,c, \varepsilon, n} = \Lambda_\text{erg} \cap \bigcup_{k \in \N} \bigcup_{n\in \N} \Lambda_{k,c, \varepsilon, n}$$
		for $\mu$-a.e.$[\gamma]\in \Gamma \backslash \Geod(\X)$. Repeating the argument for $\gamma^-$, we get that 
		$$\gamma^{\pm} \in \Lambda_\text{erg} \cap \bigcup_{k \in \N} \bigcup_{n\in \N} \Lambda_{k,c, \varepsilon, n}$$
		for $\mu$-a.e.$[\gamma]\in \Gamma \backslash \Geod(\X)$.
		
		We have just showed that
		$\mu(\Gamma \backslash \Geod(\X; \Lambda_\text{erg} \cap \bigcup_{k\in \N} \bigcup_{n\in \N} \Lambda_{k,c,\varepsilon,n}) = 1$.
		But 
		$$\Gamma \backslash \Geod\left(\X; \Lambda_\text{erg} \cap \bigcup_{k\in \N}\bigcup_{n\in \N} \Lambda_{k,c,\varepsilon,n}\right) = \bigcup_{k\in \N}\bigcup_{n\in \N} \Gamma \backslash \Geod(\X; \Lambda_{k,c,\varepsilon,n} \cap \Lambda_\text{erg}).$$
		Each set $\Gamma \backslash \Geod(\X; \Lambda_{k,c,\varepsilon,n} \cap \Lambda_\text{erg})$ is $\bar{\Phi}_1$-invariant, since its definition depends only on the values of the geodesics at infinity. Therefore, since $\mu$ is ergodic, $\mu(\Gamma \backslash \Geod(\X; \Lambda_{k,c,\varepsilon,n} \cap \Lambda_\text{erg})) \in \{0,1\}$ for every $k,n \in \N$. We conclude that there exists $k_\varepsilon,n_\varepsilon \in \N$ such that $\mu(\Gamma \backslash \Geod(\X; \Lambda_{k_\varepsilon,c, \varepsilon, n_\varepsilon} \cap \Lambda_\text{erg})) = 1$.
	\end{proof}

	We proceed to the second step, which provides an almost opposite inequality to Proposition \ref{prop:h_Lip_top_>=_MD} which holds for the sets $\Lambda_{\tau,c, \varepsilon, n} \cap \Lambda_\text{erg}$ under the additional assumptions of Theorem \ref{theo:intro-equality-Gromov-hyperbolic}.
	\begin{prop}
		\label{prop:h_Liptop_<_MD_special_subsets}
		Let $\X$ be a proper, geodesic, line-convex, Gromov-hyperbolic metric space, let $x\in \X$ and let $\Gamma < \Isom(\X)$ be discrete. For every $\tau,c,\varepsilon > 0$ and $n\in \N$ it holds that
		$$\overline{h}_f(\Lambda_{\tau,c, \varepsilon, n} \cap \Lambda_\textup{erg}) \le \overline{\MD}(\Lambda_{\tau,c, \varepsilon, n} \cap \Lambda_\textup{erg}) + h_\textup{erg}(\X;\Gamma)\cdot\frac{2\varepsilon}{c-\varepsilon}.$$
	\end{prop}
	
	This is the only point where we use the line-convexity. Actually, it will be used in the next lemma which is inspired to \cite[Lemma 4.4]{Cav21}.
	
	\begin{lemma}
		\label{lemma:key_lemma_Lip_Bow_entropy}
		Let $\X$ be a proper, geodesic, line-convex, Gromov-hyperbolic metric space, let $x\in \X$ and let $\Gamma < \Isom(\X)$ be discrete. For every $\tau,c,\varepsilon > 0$, $n\in \N$, $f\in \mathcal{F}$, $R\ge 0$ and $0<r\le r'$ it holds that
		$$\lims_{T \to +\infty} \sup_{\gamma \in \Geod(\overline{B}(x,R); \Lambda_{\tau,c,\varepsilon,n} \cap \Lambda_\textup{erg})}\frac{1}{T}\log \Cov_{\sfD_f^T}(\overline{B}_{\sfD_f^T}(\gamma,r'),r) \le h_\textup{erg}(\X;\Gamma)\cdot \frac{2\varepsilon}{c-\varepsilon}.$$
	\end{lemma}
	
	Before discussing the proof, we show how to get Proposition \ref{prop:h_Liptop_<_MD_special_subsets} from Lemma \ref{lemma:key_lemma_Lip_Bow_entropy}.
	\begin{proof}[Proof of Proposition \ref{prop:h_Liptop_<_MD_special_subsets}]
		Let $\delta$ be the iperbolicity constant of $\X$.
		In order to simplify the notation we set $C:= \Lambda_{\tau,c, \varepsilon, n} \cap \Lambda_\text{erg}$.
		By definition \eqref{eq:defin-upper-Lip-Bowen-entropy} we have that
		\begin{equation}
			\label{eq:Prop.h_Lip<h_Gamma}
			\overline{h}_f(C) = \sup_{R\ge 0} \lim_{r\to 0} \lims_{T \to +\infty} \frac{1}{T} \log \Cov_{\sfD_f^T}(\Geod(\overline{B}(x,R); C), r),
		\end{equation} 
		where $f\in \mathcal{F}$. We claim that the decreasing function
		$$r\mapsto \lims_{T \to +\infty} \frac{1}{T}\log \Cov_{\sfD_f^T}(\Geod(\overline{B}(x,R);C),r)$$
		is almost constant. For that, we fix $0<r\le r'$. Then $\frac{1}{T}\log \Cov_{\sfD_f^T}(\Geod(\overline{B}(x,R);C),r)$ is bounded from above by
		\begin{equation}
			\begin{aligned}
				&\frac{1}{T}\log \left(\Cov_{\sfD_f^T}(\Geod(\overline{B}(x,R);C),r') \cdot \sup_{\gamma \in \Geod(\overline{B}(x,R); C)} \Cov_{\sfD_f^T}(\overline{B}_{\sfD_f^T}(\gamma,r'),r)  \right)\\
				= &\frac{1}{T}\log \Cov_{\sfD_f^T}(\Geod(\overline{B}(x,R);C),r')+ \frac{1}{T}\log
				\sup_{\gamma \in \Geod(\overline{B}(x,R); C)} \Cov_{\sfD_f^T}(\overline{B}_{\sfD_f^T}(\gamma,r'),r).
			\end{aligned}
		\end{equation}
		Taking the limit superior for $T$ going to $+\infty$ we get that $\lims_{T \to +\infty}\frac{1}{T}\log \Cov_{\sfD_f^T}(\Geod(\overline{B}(x,R);C),r)$ is smaller than or equal to
		$$\lims_{T \to +\infty}\frac{1}{T}\log \Cov_{\sfD_f^T}(\Geod(\overline{B}(x,R);C),r')+ h_\text{erg}(\X;\Gamma)\cdot\frac{2\varepsilon}{c-\varepsilon},$$
		where we used Lemma \ref{lemma:key_lemma_Lip_Bow_entropy}. 
		
		We can apply this property to $r' := R + \sfd(x,\text{QC-Hull}(C)) + 44\delta + \sup_{s\in \R}2\vert s \vert f(s)$ and an arbitrary $r<r'$. By taking the limit for $r$ going to $0$ we get that 
		\begin{equation}
			\begin{aligned}
				\overline{h}_f(C) \le \sup_{R \ge 0} \lims_{T \to +\infty} \frac{1}{T} \log \Cov_{\sfD_f^T}(\Geod(\overline{B}(x,R); C), r') + h_\text{erg}(\X;\Gamma)\cdot\frac{2\varepsilon}{c-\varepsilon},
			\end{aligned}
		\end{equation}		
				
		We now claim that 
		\begin{equation}
			\lims_{T \to +\infty} \frac{1}{T} \log \Cov_{\sfD_f^T}(\Geod(\overline{B}(x,R); C), r') \le \overline{\MD}(C)
		\end{equation}
		for every $R\ge \sfd(x,\text{QC-Hull}(C)) + 22\delta =:M$.
		Let $T \ge 0$ and define $\rho := e^{-T}$. Let $\{z_i\}$ be a $\rho$-dense subset of $C$, i.e. for every $z\in C$ there exists $i$ such that $(z,z_i)_x \ge T$. Lemma \ref{lemma:approximation-ray-line} gives geodesic lines $\gamma_i \in \Geod(\overline{B}(x,M);C)$ such that $\gamma_i^+=z_i$. We claim that the set $\{\gamma_i\}$ is $r'$-dense in $\Geod(\overline{B}(x,R);C)$ with respect to $\sfD_f^T$. Let $\gamma \in \Geod(\overline{B}(x,R);C)$ and let $i$ be such that $(\gamma^+,z_i)_x \ge T$. By construction, $\sfd(\gamma(0), \gamma_i(0)) \le R+M$. Moreover, using Lemma \ref{product-rays} and twice Lemma \ref{lemma:approximation-ray-line} we get that $\sfd(\gamma(s), \gamma_i(s)) \le R + M + 22\delta$ for every $s\in [0,T]$. We estimate
		$$\sfD_f^T(\gamma,\gamma_i) = \sup_{t\in [0,T]} \sup_{s\in \R} \sfd(\gamma(s+t), \gamma_i(s+t))f(s)$$
		dividing the supremum in three parts. 
		
		If $s+t \le 0$ then $$\sfd(\gamma(s+t),\gamma_i(s+t)) \le 2\vert s+t \vert + R + M \le 2\vert s \vert + R + M,$$
		implying $\sup_{s+t \le 0} \sfd(\gamma(s+t), \gamma_i(s+t))f(s) \le \sup_{s\in \R}2\vert s \vert f(s) + R + M$. 
		
		If $0\le s+t \le T$ then
		$$\sfd(\gamma(s+t),\gamma_i(s+t)) \le R + M + 22\delta$$
		implying $\sup_{0\le s+t \le T} \sfd(\gamma(s+t), \gamma_i(s+t))f(s) \le R + M + 22\delta$.
		
		If $s+t \ge T$ then $$\sfd(\gamma(s+t),\gamma_i(s+t)) \le 2\vert s+t - T \vert + R + M + 22\delta \le 2\vert s \vert + R + M + 22\delta,$$
		implying $\sup_{s+t \ge T} \sfd(\gamma(s+t), \gamma_i(s+t))f(s) \le \sup_{s\in \R}2\vert s \vert f(s) + R + M + 22\delta$.
		
		Therefore we conclude that $\sfD_f^T(\gamma,\gamma_i) \le r'$. Hence,
		\begin{equation}
			\begin{aligned}
				\lims_{T \to +\infty} \frac{1}{T} \log \Cov_{\sfD_f^T}(\Geod(\overline{B}(x,R); C), r') \le \lims_{\rho \to 0} \frac{\log \Cov_x(C,\rho) }{\log\frac{1}{\rho}} = \overline{\MD}(C).
			\end{aligned}
		\end{equation}
	\end{proof}

	\begin{obs}
	\label{rmk:idea_behind_h_f<MD}
		The main idea in the proof of Proposition \ref{prop:h_Liptop_<_MD_special_subsets} is the following. While it is always true that $\overline{\MD}(C) \le \overline{h}_f(C)$ for every $C\subseteq \partial \X$ by Proposition \ref{prop:h_Lip_top_>=_MD}, the opposite inequality can be false in general. However, if the function 
		\begin{equation}
			\label{eq:h_f_rmk}
			r \mapsto \lims_{T \to +\infty}\frac{1}{T} \log \Cov_{\sfD_f^T}(\Geod(\overline{B}(x,R);C), r)
		\end{equation} 
		is constant then one can use a large enough value of $r$ to obtain the opposite inequality $\overline{h}_f(C) \le \overline{\MD}(C)$. This should be compared with what we did in Section \ref{subsec:geodesic_entropy_subsets}.
		
		In general there is no reason for which the function \eqref{eq:h_f_rmk} is constant. However, for the special subset $C = \Lambda_{\tau,c,\varepsilon,n} \cap \Lambda_\textup{erg}$ we get that the associated function \eqref{eq:h_f_rmk} is almost constant, with an error which is controlled by $h_\textup{erg}(\X;\Gamma)$. This is enough to get the estimate of Proposition \ref{prop:h_Liptop_<_MD_special_subsets}. The almost constancy of the function is provided by Lemma \ref{lemma:key_lemma_Lip_Bow_entropy}, which therefore plays a crucial role.
	\end{obs}
	
	In order to clarify the idea behind the proof of Lemma \ref{lemma:key_lemma_Lip_Bow_entropy} we first provide an easier version for spaces with bounded geometry.
	
	\begin{lemma}
		\label{lemma:key_lemma_Lip_Bow_entropy_doubling}
		Let $\X$ be a proper, geodesic, line-convex metric space with bounded geometry. Let $f\in \mathcal{F}$ and $0<r\le r'$. Then
		$$\lim_{T \to +\infty} \sup_{\gamma \in \Geod(\X)}\frac{1}{T}\log \Cov_{\sfD_f^T}(\overline{B}_{\sfD_f^T}(\gamma,r'),r) = 0.$$
	\end{lemma}
	\begin{proof}
		Let $\gamma \in \Geod(\X)$. Let $S >0$ such that $2\vert s \vert f(s) \le \frac{r}{4}$ for every $s\ge S$. Let $T\ge S$. Our goal is to bound from above the quantity $\Pack_{\sfD_f^T}(\overline{B}_{\sfD_f^T}(\gamma,r'),\frac{r}{2})$. 
		Let $\gamma' \in \overline{B}_{\sfD_f^T}(\gamma,r')$. By Lemma \ref{lemma-metric-locgeod} we get that $\sfd(\gamma(t), \gamma'(t)) = \sfd((\Phi_t (\gamma))(0), (\Phi_t (\gamma'))(0)) \le r'$ for every $t\in [0,T]$. The triangular inequality implies that $$\sfd(\gamma'(T + S), \gamma(T)) \le S + r', \qquad \sfd(\gamma'(-S), \gamma(0)) \le S + r'.$$
		In other words, for every $\gamma' \in \overline{B}_{\sfD_f^T}(\gamma,r')$ we have that $\gamma'(T+S) \in \overline{B}(\gamma(T), S + r')$ and $\gamma'(-S) \in \overline{B}(\gamma(0), S+r')$. Moreover, suppose that two geodesics $\gamma',\gamma'' \in \overline{B}_{\sfD_f^T}(\gamma,r')$ satisfy $\sfd(\gamma'(T + S), \gamma''(T + S)) \le \frac{r}{4}$ and $\sfd(\gamma'(-S), \gamma''(-S)) \le \frac{r}{4}$. By line-convexity of the space we deduce that $\sfd(\gamma'(s), \gamma''(s)) \le \frac{r}{4}$ for every $s\in [-S, T+ S]$. In order to estimate
		\begin{equation}
			\begin{aligned}
				\sfD_f^T(\gamma',\gamma'') &= \sup_{t\in [0,T]} \sup_{s\in \R} \sfd(\gamma'(s+t), \gamma''(s+t)) f(s),
			\end{aligned}
		\end{equation}
		we divide the second supremum in three parts. 
		
		If $s+t \le -S$ then $\sfd(\gamma'(s+t), \gamma''(s+t)) \le \frac{r}{4} + 2\vert s+t + S \vert \le \frac{r}{4} + 2\vert s \vert$, so
		$$\sup_{s+t \le -S} \sfd(\gamma'(s+t), \gamma''(s+t))f(s) \le \sup_{s+t \le -S} \left(\frac{r}{4} + 2\vert s \vert\right)f(s) \le \frac{r}{2}.$$ 
		
		If $-S \le s+T \le T+S$ then $\sfd(\gamma'(s+t), \gamma''(s+t)) \le \frac{r}{4}$ and 
		$$\sup_{-S \le s+t \le T + S} \sfd(\gamma'(s+t), \gamma''(s+t))f(s) \le \frac{r}{4}.$$ 
		
		If $s+t \ge T + S$ then $\sfd(\gamma'(s+t), \gamma''(s+t)) \le \frac{r}{4} + 2\vert s+t - T - S \vert \le \frac{r}{4} + 2\vert s \vert$, because $T + S - t \ge S \ge 0$, so
		$$\sup_{s+t \ge T + S} \sfd(\gamma'(s+t), \gamma''(s+t))f(s) \le \sup_{s+t \ge T + S} \left(\frac{r}{4} + 2\vert s \vert\right)f(s) \le \frac{r}{2},$$
		where the last inequality follows since $s\ge S$. Therefore,
		$$\sfD_f^T(\gamma',\gamma'') \le \frac{r}{2}.$$
		This means that if two elements $\gamma',\gamma''$ of $\overline{B}_{\sfD_f^T}(\gamma,r')$ are at $\sfD_f^T$-distance bigger than $\frac{r}{2}$, then at least one of the two couple of points $(\gamma'(-S),\gamma''(-S))$, $(\gamma'(T+S), \gamma''(T+S))$ are $\frac{r}{4}$-separated in $\X$. 
		This provides the bound
		$$\Pack_{\sfD_f^T}\left(\overline{B}_{\sfD_f^T}(\gamma,r'),\frac{r}{2}\right) \le \Pack_\sfd\left(\overline{B}(\gamma(0), S + r'), \frac{r}{4}\right)\cdot \Pack_\sfd\left(\overline{B}(\gamma(T), S + r'), \frac{r}{4}\right).$$
		By \eqref{eq:Cov_Pack} and the fact that $\X$ has bounded geometry we get that the quantities
		$$\Pack_\sfd\left(\overline{B}(\gamma(0), S + r'), \frac{r}{4}\right)\quad \text{ and } \quad\Pack_\sfd\left(\overline{B}(\gamma(T), S + r'), \frac{r}{4}\right)$$
		can be bound from above uniformly in $T$ and in $\gamma$. Therefore
		\begin{equation}
			\begin{aligned}
				\lims_{T \to +\infty}\sup_{\gamma \in \Geod(\X)}\frac{1}{T}\log \Cov_{\sfD_f^T}(\overline{B}_{\sfD_f^T}(\gamma,r'),r) &\le \lims_{T \to +\infty}\sup_{\gamma \in \Geod(\X)}\frac{1}{T}\log \Pack_{\sfD_f^T}\left(\overline{B}_{\sfD_f^T}(\gamma,r'),\frac{r}{2}\right)=0.
			\end{aligned}
		\end{equation}
	\end{proof}
	
	\begin{obs}
		The line-convexity assumption in Lemma \ref{lemma:key_lemma_Lip_Bow_entropy_doubling} ensures the following property. In order to have that $\gamma, \gamma' \in \Geod(\X)$ are close on the interval $[0,T]$ it is enough to test their closeness at two times: at time $0$ and at time $T$. This property fails for non line-convex spaces, like the example of Theorem \ref{theo:intro_counterexample_h_top>h_Gamma} where its lackness gives origin to additional topological entropy.
	\end{obs}

	\begin{obs}
		Following the discussion of Remark \ref{rmk:idea_behind_h_f<MD} we observe that as a consequence of Lemma \ref{lemma:key_lemma_Lip_Bow_entropy_doubling} we have the following fact. Let $\X$ be a proper, geodesic, line-convex, Gromov-hyperbolic metric space with bounded geometry. Then for every $C\subseteq \partial \X$ we have that $\overline{h}_f(C) = \overline{\MD}(C)$. This should be compared with the results of \cite[\S 6]{Cav21}.
	\end{obs}
	
	Observe that in the proof of Lemma \ref{lemma:key_lemma_Lip_Bow_entropy_doubling}, the fact that the metric space has bounded geometry has been used only once, namely to bound the cardinality of a separated set inside balls of fixed radius but different centers. The idea for Lemma \ref{lemma:key_lemma_Lip_Bow_entropy} is to obtain such a bound using the fact that every geodesic in $\Lambda_{\tau,c, \varepsilon, n} \cap \Lambda_\text{erg}$ is at controlled distance to a point of $\Gamma x$ and to reduce the packing problem to balls with fixed center. The price to pay is that the balls have radii going to $+\infty$, and that is where the condition on the ergodic entropy of the action of $\Gamma$ on $\X$ plays a role. 
	
	\begin{proof}[Proof of Lemma \ref{lemma:key_lemma_Lip_Bow_entropy}]
		In order to simplify the notations we set $C:= \Lambda_{\tau,c, \varepsilon, n} \cap \Lambda_\text{erg}$.
		Let $\gamma \in \Geod(\overline{B}(x,R); C)$. Let $S>0$ be such that $2\vert s \vert f(s) \le \frac{r}{4}$ for every $s\ge S$. Let $T \ge \max\{(c+\varepsilon)n, S\}$. The same computations of the proof of Lemma \ref{lemma:key_lemma_Lip_Bow_entropy_doubling} show that $\Pack_{\sfD_f^T}\left(\overline{B}_{\sfD_f^T}(\gamma,r'),\frac{r}{2}\right)$ is bounded from above by
		\begin{equation}
			\begin{aligned}
				 \Pack_{\sfd}\left(\overline{B}(\gamma(0), S + r') \cap \text{QC-Hull}(C), \frac{r}{4}\right) \cdot \Pack_{\sfd}\left(\overline{B}(\gamma(T), S + r') \cap \text{QC-Hull}(C), \frac{r}{4}\right).
			\end{aligned}
		\end{equation}
		The quantity $\Pack_{\sfd}\left(\overline{B}(\gamma(0), S + r') \cap \text{QC-Hull}(C), \frac{r}{4}\right)$ is bounded uniformly from above with respect to $\gamma$ because $\X$ is proper and $\gamma(0)$ belongs to the fixed compact set $\overline{B}(x,R)$.
		
		It remains to bound $\Pack_{\sfd}\left(\overline{B}(\gamma(T), S + r') \cap \text{QC-Hull}(C), \frac{r}{4}\right)$ from above. Let $z = \gamma^+ \in \partial \X$ and let $\xi_{x,z}$ be a geodesic ray appearing in the condition $z\in\Lambda_{\tau,c, \varepsilon, n}$. More precisely, the following conditions are true:
		\begin{itemize}
			\item[(i)] $(c-\varepsilon)i \le \vartheta_i \le (c+\varepsilon)i$ for every $i\ge n$.
			\item[(ii)] For every $i$ there exists $t_i \in [\vartheta_i, \vartheta_{i+1}]$ such that $\sfd(\xi_{x,z}(t_i), \Gamma x) \le \tau$.
		\end{itemize}
	By assumptions, we restricted to the values $T\ge (c+\varepsilon)n \ge \vartheta_n$. Then there exists $i_T \ge n$ such that $T\in [\vartheta_{i_T},\vartheta_{i_T + 1}]$ and we have that 
	$$\sfd(\xi_{x,z}(T), \Gamma x) \le \vert T - t_{i_T} \vert + \tau \le \vartheta_{i_T + 1} - \vartheta_{i_T} + \tau \le 2\varepsilon i_T + c + \varepsilon + \tau \le \frac{2\varepsilon}{c-\varepsilon} T + c + \varepsilon + \tau.$$ 
	By Lemma \ref{parallel-geodesics} we know that $\sfd(\gamma(T), \xi_{x,z}(T)) \le R +8\delta$, so $\sfd(\gamma(T), \Gamma x) \le \frac{2\varepsilon}{c-\varepsilon} T + c + \varepsilon + \tau + R + 8\delta$.
	We can now estimate 
	\begin{equation}
		\begin{aligned}
			\Pack_{\sfd}&\left(\overline{B}(\gamma(T), S + r') \cap \text{QC-Hull}(C), \frac{r}{4}\right)\\&\le \Pack_{\sfd}\left(\overline{B}\left(x, S + r' + \frac{2\varepsilon}{c-\varepsilon} T + c + \varepsilon + \tau + R + 8\delta\right) \cap \text{QC-Hull}(\Lambda_\text{erg}), \frac{r}{4}\right).
		\end{aligned}
	\end{equation}
	Observe that in this inequality we replaced the set $C$ with the larger set $\Lambda_\text{erg}$, which is $\Gamma$-invariant. In particular the set $\text{QC-Hull}(\Lambda_\text{erg})$ is $\Gamma$-invariant too and we can translate our counting problem to a counting problem on the ball centered at $x$.

	By definition of $h_\text{erg}(\X;\Gamma)$ we get that
	\begin{equation}
		\begin{aligned}
			&\lims_{T \to +\infty}\sup_{\gamma \in \Geod(\overline{B}(x,R); \Lambda_{\tau,c,\varepsilon,n} \cap \Lambda_\text{erg})} \frac{1}{T} \log \Pack_{\sfD_f^T}\left(\overline{B}_{\sfD_f^T}(\gamma,r'),\frac{r}{2}\right) \\
			\le &\lims_{T \to +\infty} \frac{1}{T} \log \Pack_{\sfd}\left(\overline{B}\left(x, S + r' + \frac{2\varepsilon}{c-\varepsilon} T + c + \varepsilon + \tau + R + 8\delta\right) \cap \text{QC-Hull}( \Lambda_\text{erg}), \frac{r}{4}\right)\\
			\le & h_\text{erg}(\X;\Gamma)\cdot\frac{2\varepsilon}{c-\varepsilon}.
		\end{aligned}
	\end{equation}
	\end{proof}

	We have all the ingredients to give the
	\begin{proof}[Proof of Theorem \ref{theo:intro-equality-Gromov-hyperbolic}]
	In view of Theorem \ref{theo:intro-h_top>=h_Gamma} and the discussion at the beginning of the section, it is enough to show that \eqref{eq:strategy_for_h_top_<h_Gamma} holds. So, we fix $\mu \in \mathcal{E}_1(\Gamma \backslash
	 \Geod(\X), \bar{\Phi}_1)$. Proposition \ref{prop:Birkhoff_mean} gives $c \in [1,+\infty)$ such that, for every $\varepsilon > 0$ there exist $\tau_\varepsilon >0$ and $n_\varepsilon\in \N$ such that $\mu(\Gamma \backslash \Geod(\X; \Lambda_{\tau_\varepsilon,c, \varepsilon, n_\varepsilon} \cap \Lambda_\text{erg})) = 1$. The relations
	 $$h_\mu(\Gamma \backslash \Geod(\X), \bar{\Phi}_1) = h_\mu(\Gamma \backslash \Geod(\X;\Lambda_{\tau_\varepsilon,c,\varepsilon, n_\varepsilon} \cap \Lambda_\text{erg}), \bar{\Phi}_1) \le h_\text{top}(\Gamma \backslash \Geod(\X;\Lambda_{\tau_\varepsilon,c,\varepsilon,n_\varepsilon} \cap \Lambda_\text{erg}), \bar{\Phi}_1)$$
	 are true because of Lemma \ref{lemma:measure_entropy_on_measure_1} and Proposition \ref{prop:variational_principle}. We now prove that
	 $$h_\text{top}(\Gamma \backslash \Geod(\X;\Lambda_{\tau_\varepsilon,c,\varepsilon,n_\varepsilon} \cap \Lambda_\text{erg}), \bar{\Phi}_1) \le \overline{h}_f(\Lambda_{\tau_\varepsilon,c,\varepsilon, n_\varepsilon} \cap \Lambda_\text{erg}),$$
	 where $f\in \mathcal{F}$. Since $\Gamma$ acts discretely and by isometries on $(\Geod(\X), \sfD_f)$, by Proposition \ref{prop:quotient_of_loc_geod}, then $\sfD_f$ induces the quotient metric $\bar{\sfD}_f$ on $\Gamma \backslash\Geod(\X)$ that restricts to a metric on $\Gamma \backslash\Geod(\X; \Lambda_{\tau_\varepsilon,c, \varepsilon, n_\varepsilon} \cap \Lambda_\text{erg})$. Let $K \subseteq \Gamma \backslash \Geod(\X; \Lambda_{\tau_\varepsilon,c, \varepsilon, n_\varepsilon} \cap \Lambda_\text{erg})$ be compact. We claim that there exists $R\ge 0$ such that $K \subseteq \Gamma \backslash \Geod(\overline{B}(x,R); \Lambda_{\tau_\varepsilon,c, \varepsilon, n_\varepsilon} \cap \Lambda_\text{erg})$. Indeed, if not, for every $j \in \N$ we could find $[\gamma_j] \in K$ such that $\sfd(\Gamma \gamma_j(0), x) > j$. Up to a subsequence we have that $[\gamma_j]$ converges to $[\gamma_\infty]$. But this means that there are $g_j \in \Gamma$ such that $g_j\gamma_j$ converges to $\gamma_\infty$. In particular $\sfd(\Gamma \gamma_j(0), \gamma_\infty(0))$ tends to zero, which is a contradiction. Therefore we can bound $h_\text{top}(\Gamma \backslash \Geod(\X;\Lambda_{\tau_\varepsilon,c,\varepsilon,n_\varepsilon} \cap \Lambda_\text{erg}), \bar{\Phi}_1)$ from above by
	 \begin{equation}
	 	\begin{aligned}
	 		&\sup_{R \ge 0} \lim_{r\to 0} \lims_{n \to +\infty} \frac{1}{n}\log \Cov_{\bar{\sfD}_f^n}(\Gamma \backslash \Geod(\overline{B}(x,R); \Lambda_{\tau_\varepsilon,c, \varepsilon, n_\varepsilon} \cap \Lambda_\text{erg}), r) \\
	 		\le &\sup_{R \ge 0} \lim_{r\to 0} \lims_{n \to +\infty} \frac{1}{n}\log \Cov_{\sfD_f^n}(\Geod(\overline{B}(x,R); \Lambda_{\tau_\varepsilon,c, \varepsilon, n_\varepsilon} \cap \Lambda_\text{erg}), r),
	 	\end{aligned}
	 \end{equation}
 	because the projection map $p\colon (\Geod(\X), \sfD_f) \to (\Gamma \backslash \Geod(\X), \bar{\sfD}_f)$ is $1$-Lipschitz. Indeed, for every $n\in \N$ we have that
 	\begin{equation}
 		\begin{aligned}
 			\bar{\sfD}_f^n([\gamma],[\gamma']) = \sup_{i=0,\ldots,n} \bar{\sfD}_f(\bar{\Phi}_i([\gamma]),\bar{\Phi}_i([\gamma'])) &= \sup_{i=0,\ldots,n} \bar{\sfD}_f([\Phi_i(\gamma)],[\Phi_i(\gamma')])\\&\le \sup_{i=0,\ldots,n} {\sfD}_f(\Phi_i(\gamma),\Phi_i(\gamma')) = \sfD_f^n(\gamma, \gamma')
 		\end{aligned}
 	\end{equation}
 	for every $\gamma,\gamma'\in \Geod(\X)$ and every $n\in \N$.
 	
 	By taking the limit superior over all real numbers instead of the integers, we obtain that
 	$$h_\text{top}(\Gamma \backslash \Geod(\X;\Lambda_{\tau_\varepsilon,c,\varepsilon,n_\varepsilon} \cap \Lambda_\text{erg}), \bar{\Phi}_1) \le \overline{h}_{f}(\Lambda_{\tau_\varepsilon,c,\varepsilon,n_\varepsilon} \cap \Lambda_\text{erg}).$$
 	Finally, Propositions \ref{prop:h_Liptop_<_MD_special_subsets} and \ref{prop:MD < h_Gamma_special_subsets} give that
 	$$h_\mu(\Gamma \backslash \Geod(\X), \bar{\Phi}_1) \le h_\text{crit}(\X;\Gamma)\cdot\frac{c+\varepsilon}{c-\varepsilon} + h_\text{erg}(\X;\Gamma)\cdot \frac{2\varepsilon}{c-\varepsilon}.$$
 	By taking the limit for $\varepsilon$ going to zero we deduce that
 	$$h_\mu(\Gamma \backslash \Geod(\X), \bar{\Phi}_1) \le h_\text{crit}(\X;\Gamma).$$
	\end{proof}
	
	\section{The proofs of Theorems \ref{theo:intro_counterexample_h_top>h_Gamma}, \ref{theo:intro-counter-tree}, \ref{theo:intro-weird-example}}
	
	In this section we will present the proofs of the Theorems \ref{theo:intro_counterexample_h_top>h_Gamma}, \ref{theo:intro-counter-tree} and \ref{theo:intro-weird-example}. They will be mainly based on the computation of the topological entropy of the following model example. Let $F_\ell = \langle a_1,\ldots, a_\ell \rangle$ be the free group on $\ell$ generators. Let $T_{2\ell}$ be the regular tree of valency $2\ell$, which is also the Cayley graph of $F_\ell$. The action of $F_\ell$ on the geodesically complete, $\CAT(-1)$ space $T_{2\ell}$ is discrete and cocompact.
	\begin{lemma}
		\label{lemma:crit_F_ell}
		We have that $h_{\textup{crit}}(T_{2\ell};F_\ell) = \log(2\ell-1)$.
	\end{lemma}
	\begin{proof}
		Arguing as in \cite[Proposition 5.7]{Cav21ter} one gets that $h_{\textup{crit}}(T_{2\ell};F_\ell)$ coincides with the so called covering entropy of $T_{2\ell}$, and so with every notion of entropy of $T_{2\ell}$ studied in \cite{Cav21}. Choosing the volume entropy with respect to the $1$-dimensional Hausdorff measure $\mathcal{H}^1$, we get that
		$$h_{\textup{crit}}(T_{2\ell};F_\ell) = \lims_{T \to +\infty} \frac{1}{T}\log \mathcal{H}^1(B(e,T)).$$
		Since $\mathcal{H}^1(B(e,T)) = 2\ell\cdot (2\ell-1)^{\lfloor T- 1 \rfloor}(1 + (2\ell-1)(T - \lfloor T \rfloor))$ we get that
		$$h_{\textup{crit}}(T_{2\ell};F_\ell) = \log (2\ell-1).$$
	\end{proof}

	We compute explicitily the topological entropy of the dynamical system $(F_\ell \backslash \Geod(T_{2\ell}), \bar{\Phi}_1)$, which equals $\log(2\ell-1)$ by Theorem \ref{theo:intro-equality-Gromov-hyperbolic} and Lemma \ref{lemma:crit_F_ell}. Indeed this computation serves as a model for similar ones that will be performed later.
	
	\begin{prop}
		\label{prop:computation_h_top_F_ell}
		We have that $h_\textup{top}(F_\ell \backslash \Geod(T_{2\ell}), \bar{\Phi}_1) = \log(2\ell-1)$.
	\end{prop}
	
	\begin{proof}
		We need to parametrize the space $F_\ell \backslash \Geod(T_{2\ell})$. Let $\Sigma$ be the set of symbols $\Sigma := \{a_1^{\pm 1}, \ldots, a_\ell^{\pm 1}\}$. Let $\mathcal{A}_\Z(\Sigma)$ be the set of sequences $w = (w_i)_{i\in\Z}$, with $w_i \in \Sigma$. To every $w \in \mathcal{A}_\Z(\Sigma)$ we can associate a curve of $T_{2\ell}$ in the following way. Remember that $T_{2\ell}$ is the Cayley graph of the group $F_\ell = \langle a_1,\ldots,a_\ell\rangle$.	
		Then each symbol $a_i^{\pm 1} \in \Sigma$ is associated to the geodesic segment that joins $e$ with the vertex $a_i^{\pm 1}$. The curve associated to $w = (w_i)_{i\in \Z}$ is the curve associated to the infinite word $w$ as element of $F_\ell$.
		We call this curve $\gamma_w$. Notice that $\gamma_w$ is a geodesic if and only if $w$ belongs to the set $\mathcal{A}_\Z^\text{red}(\Sigma) := \{(w_i)_{i\in \Z} \in \mathcal{A}_\Z(\Sigma)\,:\, w_{i+1} \neq w_i^{-1} \text{ for every } i \in \Z\}$. With 'red' we mean reduced, since these sequences are exactly the reduced words in $F_\ell$.
		
		Moreover, every geodesic $\gamma$ of $T_{2\ell}$ such that $\gamma(0) = e$ is of the form $\gamma = \gamma_w$ for some $w\in \mathcal{A}_\Z^\text{red}(\Sigma)$. Since $F_\ell$ acts transitively on the vertices of $T_{2\ell}$ we deduce the following. Every $\gamma \in \Geod(T_{2\ell})$ is of the form $\gamma = g\Phi_t (\gamma_w)$ for some $t\in [0,1]$, $w \in \mathcal{A}_\Z^\text{red}(\Sigma)$ and $g\in F_\ell$. This implies that every element $[\gamma] \in F_\ell \backslash \Geod(T_{2\ell})$ is of the form $\bar{\Phi}_t([\gamma_w])$ for some $t\in [0,1]$ and $w \in \mathcal{A}_\Z^\text{red}(\Sigma)$.
		
		Recall that the dynamical system $(F_\ell \backslash \Geod(T_{2\ell}), \bar{\Phi}_1)$ is compact, so we can use the metric $\bar{\sfD}_f$, with $f(s) = e^{-\vert s \vert}\in \mathcal{F}$, to compute its topological entropy, by Proposition \ref{prop:variational_principle}. 
		For every $k < m\in \Z$ we consider the set
		$\mathcal{A}_{k,m}^\text{red}(\Sigma) = \{ (v_i)_{i=k,\ldots,m}\,:\, v_i \in \Sigma \text{ and } v_{i+1} \neq v_i^{-1} \text{ for every } i=k,\ldots,m-1\}$. We fix $0<r<1$ and $n \in \N$ and we choose $k_r \in \N$ such that $\sup_{s\ge k_r-1} se^{-s} < \frac{r}{8}$. For every element of $v = (v_i) \in \mathcal{A}_{-k_r,n+k_r}^\text{red}(\Sigma)$ we choose one element $w_v = (w_i) \in \mathcal{A}_\Z^\text{red}(\Sigma)$ such that $v_i = w_i$ for every $i=-k_r,\ldots,n+k_r$. We claim that the set  $$\mathcal{B}_{r,n} := \bigcup_{j=1}^{\frac{2}{r}} \bar{\Phi}_{\frac{rj}{2}}(\{[\gamma_{w_v}] \,:\, v \in  \mathcal{A}^\text{red}_{-k_r,n+k_r}(\Sigma)\})$$
		is $r$-dense in $F_\ell \backslash \Geod(T_{2\ell})$ with respect to the metric $\bar{\sfD}_{f}^n$. For that, we first compute
		\begin{equation}
			\label{eq:distance_gamma_Phi_gamma}
			\bar{\sfD}_{f}^n([\gamma], \bar{\Phi}_t([\gamma])) = \max_{i=0,\ldots,n} \inf_{g\in F_\ell} \sup_{s\in \R} \sfd(g\gamma(i+s), \gamma(i+t+s))e^{-\vert s \vert} = \max_{i=0,\ldots,n} \sup_{s\in \R} te^{-\vert s \vert} = t
		\end{equation}
		for every $\gamma \in \Geod(T_{2\ell})$ and every $t\le\frac12$.
		Here we used that the systole of $F_\ell$ on $T_{2\ell}$ is $1$. Let now $[\gamma] \in F_\ell \backslash \Geod(T_{2\ell})$ be arbitrary. We know that there exists $t\in [0,1]$ and $w\in \mathcal{A}_\Z^\text{red}(\Sigma)$ such that $[\gamma] = \bar{\Phi}_t ([\gamma_w])$. We find $j \in \{1,\ldots \frac{2}{r}\}$ such that $\vert t - \frac{rj}{2}\vert < \frac{r}{2}$ and we define $[\gamma'] := \bar{\Phi}_{\frac{rj}{2}} ([\gamma_w])$. By \eqref{eq:distance_gamma_Phi_gamma} we get that $\bar{\sfD}_{f}^n([\gamma], [\gamma']) < \frac{r}{2}$. Let now $v = (w_{-k_r}, \ldots, w_{n+k_r}) \in \mathcal{A}_{-k_r, n+k_r}^\text{red}(\Sigma)$ and $w_v \in \mathcal{A}^\text{red}_\Z(\Sigma)$ be its associated element. By definition, $\bar{\Phi}_{\frac{rj}{2}} ([\gamma_{w_v}]) \in \mathcal{B}_{r,n}$. If we prove that $\bar{\sfD}_{f}^n([\gamma'], \bar{\Phi}_{\frac{rj}{2}}( [\gamma_{w_v}])) < \frac{r}{2}$ we conclude the proof of the claim, since then $\bar{\sfD}_{f}^n([\gamma], \bar{\Phi}_{\frac{rj}{2}}( [\gamma_{w_v}])) < r$. Observe that, by definition, $\gamma'(t) = \Phi_{\frac{rj}{2}}(\gamma_{w_v})(t)$ for every $t \in [-k_r + 1, n+k_r]$. On the other hand for every $t$ we have $\sfd(F_\ell\gamma'(t), \Phi_{\frac{rj}{2}}(\gamma_{w_v})(t)) \le 1$ because the diameter of $F_\ell \backslash T_{2\ell}$ is $1$. 
		Therefore a raw estimate gives that
		\begin{equation}
			\begin{aligned}
				\bar{\sfD}_{f}^n\left([\gamma'], \bar{\Phi}_{\frac{rj}{2}} ([\gamma_{w_v}])\right) &\le \max_{i=0,\ldots,n} \sup_{s\in \R} \sfd\left(\gamma'(s+i), \Phi_{\frac{rj}{2}}(\gamma_{w_v})(s+i)\right)e^{-\vert s \vert}\\
				&\le \max_{i=0,\ldots,n} \sup_{s\in (-\infty, -k_r + 1 - i]} 2\vert s \vert e^{-\vert s \vert} + \sup_{s\in [n + k_r - i, + \infty)} 2\vert s \vert e^{-\vert s \vert}\\
				&\le \sup_{s\in [k_r-1, +\infty)} 4\vert s \vert e^{- s} < \frac{r}{2},
			\end{aligned}
		\end{equation}
		because of the choice of $k_r$. Hence we proved that the set $\mathcal{B}_{r,n}$ is $r$-dense in $F_\ell \backslash \Geod(T_{2\ell})$. So,
		\begin{equation}
			\begin{aligned}
				\textup{Cov}_{\bar{\sfD}_{f}^n}(F_\ell \backslash \Geod(T_{2\ell}),r) \le \#\mathcal{B}_{r,n} = \frac{2}{r} \# \mathcal{A}_{-k_r,n+k_r}^\text{red}(\Sigma) = \frac{2}{r} \cdot (2\ell) \cdot (2\ell - 1)^{n+2k_r - 1}.
			\end{aligned}
		\end{equation}
		Therefore 
		$$h_\text{top}(F_\ell \backslash \Geod(T_{2\ell}), \bar{\Phi}_1) \le \lim_{r\to 0} \lims_{n \to +\infty} \frac{1}{n} \log\left( \frac{2}{r} \cdot (2\ell) \cdot (2\ell - 1)^{n+2k_r-1} \right) = \log (2\ell - 1).$$
		
		It remains to bound $h_\text{top}(F_\ell \backslash \Geod(T_{2\ell}), \bar{\Phi}_1)$ from below. For that, we claim that the set 
		$$\mathcal{B}_{n}' := \{[\gamma_{w_v}]\,:\, v \in \mathcal{A}_{-n+1,0}^\text{red}(\Sigma)\}$$
		is $\frac{1}{3}$-separated with respect to the metric $\bar{\sfD}_{f}^n$. Indeed, for every two distinct elements $[\gamma_{w_v}], [\gamma_{w_{v'}}] \in \mathcal{B}_n'$ there exists $i \in \{-n+1, \ldots, 0\}$ such that $v_i \neq v_i'$. Then
		$$\bar{\sfD}_{f}^n([\gamma_{w_v}], [\gamma_{w_{v'}}]) \ge \inf_{g\in F_\ell} {\sfD}_{f}(g\Phi_i(\gamma_{w_v}), \Phi_i (\gamma_{w_{v'}})) \ge \frac{1}{2}e^{-\frac{1}{4}} > \frac{1}{3}.$$
		Here we used that $\inf_{g\in F_\ell} \sfd(g\gamma_{w_v}(i+\frac14), \gamma_{w_{v'}}(i+\frac14)) = \frac12$.
		Therefore, if $r<\frac{1}{3}$ we necessarily have that
		\begin{equation}
			\begin{aligned}
				\textup{Cov}_{\bar{\sfD}_{f}^n}(F_\ell \backslash \Geod(T_{2\ell}),r) \ge \#\mathcal{B}_{n}' = \# \mathcal{A}_{-n + 1, 0}^\text{red}(\Sigma) = (2\ell) \cdot (2\ell - 1)^{n - 2},
			\end{aligned}
		\end{equation}
		and so
		$$h_\text{top}(F_\ell \backslash \Geod(T_{2\ell}), \bar{\Phi}_1) \ge \lim_{r\to 0} \lims_{n \to +\infty} \frac{1}{n} \log\left( (2\ell) \cdot (2\ell - 1)^{n - 2} \right) = \log (2\ell-1).$$
	\end{proof}

	Building on these computations we will provide the proofs of the remaining three theorems.
	\begin{proof}[Proof of Theorem \ref{theo:intro_counterexample_h_top>h_Gamma}]
		Let $\ell \in \N$ be fixed. We consider the metric space $\X$ built in this way. We start with $T_{2\ell}$ and for every couple of adjacent vertices we glue another edge of length $1$. Then $\X$ is proper, geodesic, geodesically complete and metrically doubling up to scale $1$. It is also Gromov-hyperbolic being at bounded Hausdorff distance from $T_{2\ell}$, so quasi-isometric to it. However, it is not line-convex. As group of isometries we take $F_\ell$ that acts preserving its action on $T_{2\ell}$ and sending added edges to the corresponding added edges. This action on $\X$ is by isometries, discrete, cocompact, non-elementary and of course $F_\ell$ is torsion-free. Arguing as in Lemma \ref{lemma:crit_F_ell} we get that $h_\text{crit}(\X;F_\ell) = \log(2\ell - 1)$. But, on the other hand, the topological entropy of $(F_\ell \backslash \Geod(\X), \bar{\Phi}_1)$ equals $\log(4\ell - 2)$. This follows by the same computations of Proposition \ref{prop:computation_h_top_F_ell}, once we parametrize the space $F_\ell \backslash \Geod(\X)$. In this case, for every $a_i^{\pm 1} \in F_\ell$ we denote by $\tilde{a}_i^{\pm 1}$ the geodesic segment which was added to $T_{2\ell}$ in the definition of $\X$ and which joins the same endpoints of the geodesic segment associated to $a_i^{\pm 1}$. Then, with the notation of Proposition \ref{prop:computation_h_top_F_ell} one has to use the set of symbols $\Sigma := \{a_1^{\pm 1}, \ldots, a_\ell^{\pm 1}, \tilde{a}_1^{\pm 1}, \ldots, \tilde{a}_\ell^{\pm 1}\}$. Moreover every geodesic $\gamma \in \Geod(\X)$ with $\gamma(0) = e$ is of the form $\gamma = \gamma_w$ for some $w \in \mathcal{A}_\Z^\text{red}(\Sigma)$, where now
		$\mathcal{A}_\Z^\text{red}(\Sigma) := \{(w_i)_{i\in \Z} \in \mathcal{A}_\Z(\Sigma)\,:\, w_{i+1} \neq w_i^{-1},\tilde{w}_i^{-1}\}$. The rest of the computations are identical and lead to the equality $h_\text{top}(F_\ell \backslash \Geod(\X), \bar{\Phi}_1) = \log(4\ell - 2)$.
	\end{proof}

	We continue with the
	\begin{proof}[Proof of Theorem \ref{theo:intro-counter-tree}]
		We think again to $T_4$ as the Cayley graph of $F_2 = \langle a_1,a_2 \rangle$. Let $\tilde{a}_1, \tilde{a}_2$ be the axes of the isometries $a_1$ and $a_2$. We consider the standard topological immersion of $T_4$ in $\R^2$ with edges labelled by $a_1$ parallel to the $x$-axis, and edges labelled by $a_2$ parallel to the $y$-axis.
		We denote by $R_\theta$ the Euclidean rotation of angle $\theta$. Then $R_{\frac{\pi}{2}}$ defines an isometric automorphism of $T_4$. We set $\Gamma := \langle a_1,a_2, R_{\frac{\pi}{2}} \rangle$. A direct computation shows that $R_{\frac{\pi}{2}} a_1 R_{\frac{\pi}{2}}^{-1} = a_2$ and $R_{\frac{\pi}{2}} a_2 R_{\frac{\pi}{2}}^{-1} = a_1$, so $F_2 \triangleleft \Gamma$ and $[\Gamma : F_2] \le \#\langle R_{\frac{\pi}{2}}\rangle = 4$. In particular $\Gamma < \Isom(T_4)$ is discrete, cocompact and virtually torsion-free.
		Theorem \ref{theo:intro-equality-Gromov-hyperbolic}, Lemma \ref{lemma:reduction_torsion_free} and Lemma \ref{lemma:crit_F_ell} imply that 
		$$h_\text{top}(\Gamma \backslash \Geod(T_4), \bar{\Phi}_1) = h_\text{crit}(T_4;\Gamma) = h_\text{crit}(T_4; F_2) = \log 3.$$
		
		On the other hand we need to consider the action of $\Gamma / F_2$ induced on $\LocGeod{F_2 \backslash T_4}$ by the action of $\Gamma / F_2$ on $F_2\backslash T_4 = \mathbb{S}^1 \ast \mathbb{S}^1$, the wedge of two circles of length $1$. Every local geodesic $\gamma$ of $\mathbb{S}^1\ast\mathbb{S}^1$ with $\gamma(0)=[e]$ is parametrized by a word in $\mathcal{A}_\Z^\text{red}(\Sigma)$, where $\Sigma = \{a_1^{\pm 1}, a_2^{\pm 1}\}$ as defined in the proof of Proposition \ref{prop:computation_h_top_F_ell}. The action of $\Gamma \backslash F_2 = \langle R_{\frac{\pi}{2}} \rangle$ on $\mathbb{S}^1\ast \mathbb{S}^1$ identifies all local geodesic segments associated to the symbols $a_1^{\pm 1}, a_2^{\pm 1}$. Therefore every element $[\gamma]$ of $(\Gamma / F_2) \backslash \LocGeod{\mathbb{S}^1\ast \mathbb{S}^1}$ is of the form $\bar{\Phi}_t([\gamma_w])$ where $w \in \mathcal{A}_\Z(\Sigma') = $ and $\Sigma' = \{a\}$. The same computations of Proposition \ref{prop:computation_h_top_F_ell} show that
		$$h_\text{top}((\Gamma / F_2) \backslash \LocGeod{\mathbb{S}^1 \ast \mathbb{S}^1}, \bar{\Phi}_1) = 0.$$
	\end{proof}
	
	We conclude with the
	\begin{proof}[Proof of Theorem \ref{theo:intro-weird-example}]
		We start with $T_4$ and we glue one $\mathbb{S}^1$ of length $1$ by one point at every vertex of $T_4$. We denote by $\X$ the resulting graph. It is proper, geodesic and Gromov-hyperbolic since it is at finite Hausdorff distance from $T_4$, so quasi-isometric to it. It is also line-convex because every geodesic line of $\X$ lies entirely inside $T_4$. 
		The action of $F_2$ on $T_4$ extends to a isometric, discrete, non-elementary and cocompact action of $F_2$ on $\X$. By Theorem \ref{theo:intro-equality-Gromov-hyperbolic} and Lemma \ref{lemma:crit_F_ell} we have that
		$$h_\text{top}(F_2 \backslash \Geod(\X), \bar{\Phi}_1) = h_\text{crit}(\X;F_2) = h_\text{crit}(T_4;F_2) = \log 3.$$
		The second equality follows from the fact that the $F_2$ orbit of a vertex is contained in $T_4$.
		
		On the other hand, the space $F_2 \backslash \X$ is isometric to the wedge of three $\mathbb{S}^1$ of length $1$ by one common point. It is a locally $\CAT(-1)$, locally geodesically complete space. The space $\X$ is not the universal cover of $F_2 \backslash \X$ because it is not simply connected. Indeed the universal cover of $F_2 \backslash \X$ is $T_6$ and the fundamental group of $F_2\backslash \X$ is $F_3$. Therefore, again by Theorem \ref{theo:intro-equality-Gromov-hyperbolic}, Lemma \ref{lemma:crit_F_ell} and Corollary \ref{cor:quotient_geodesics=local_geodesics_of_quotient} we have that
		$$h_\text{top}(\LocGeod{F_2\backslash \X}, \Phi_1) = h_\text{top}(F_3 \backslash \Geod(T_6), \bar{\Phi}_1) = h_\text{crit}(T_6; F_3) = \log 5.$$
	\end{proof}

	\bibliographystyle{alpha}
	\bibliography{Otal-Peigne-for-CAT-spaces}
	
\end{document}